\documentclass[a4paper,12pt]{amsart}
\usepackage[T1]{fontenc}
\usepackage[english]{babel}
\usepackage[left=1in,right=1in,top=1.1in,bottom=1.1in]{geometry}

\allowdisplaybreaks

\usepackage{mathtools}
\usepackage{times}
\usepackage[nopatch=footnote]{microtype}
\usepackage{mathrsfs,amssymb,array,multirow,comment,pifont}
\usepackage{enumitem}
\usepackage{tikz-cd}

\usepackage[dvipsnames,svgnames,table]{xcolor}
\definecolor{red}{RGB}{255,25,25}
\definecolor{blue}{RGB}{25,50,200}

\usepackage{thmtools}   

\usepackage[pagebackref,linktocpage]{hyperref}

\hypersetup{
hypertexnames=false,
pdftitle={},
pdfauthor={Hu--Jiang},	
colorlinks=true,			
linkcolor=red,			
citecolor=MidnightBlue,	
filecolor=magenta,		
urlcolor=red	
}

\usepackage{cleveref}	

\newtheorem{theorem}{Theorem}[section]
\crefname{theorem}{Theorem}{Theorems}
\newtheorem{lemma}[theorem]{Lemma}
\crefname{lemma}{Lemma}{Lemmas}
\newtheorem{proposition}[theorem]{Proposition}
\crefname{proposition}{Proposition}{Propositions}

\crefname{prop}{Proposition}{Propositions}
\newtheorem{corollary}[theorem]{Corollary}
\crefname{corollary}{Corollary}{Corollaries}

\crefname{cor}{Corollary}{Corollaries}

\crefname{conjecture}{Conjecture}{Conjectures}

\crefname{conj}{Conjecture}{Conjectures}
\newtheorem*{conj*}{Conjecture}
\crefname{conj}{Conjecture}{Conjectures}

\crefname{thmx}{Theorem}{Theorems}

\theoremstyle{definition}
\newtheorem{definition}[theorem]{Definition}
\crefname{definition}{Definition}{Definitions}

\crefname{defn}{Definition}{Definitions}
\newtheorem{example}[theorem]{Example}
\crefname{example}{Example}{Examples}

\crefname{notation}{Notation}{Notation}
\newtheorem*{notation*}{Notation}
\crefname{notation}{Notation}{Notation}
\newtheorem*{convention*}{Convention}
\crefname{convention}{Convention}{Convention}

\crefname{problem}{Problem}{Problems}

\crefname{question}{Question}{Questions}

\crefname{condition}{Condition}{Conditions}

\crefname{assumption}{Assumption}{Assumptions}

\theoremstyle{remark}

\crefname{rmk}{Remark}{Remarks}
\newtheorem*{rmk*}{Remark}
\crefname{rmk}{Remark}{Remarks}
\newtheorem{remark}[theorem]{Remark}
\crefname{remark}{Remark}{Remarks}

\crefname{fact}{Fact}{Facts}

\crefname{claim}{Claim}{Claims}
\newtheorem*{claim*}{Claim}
\crefname{claim}{Claim}{Claims}

\crefname{step}{Step}{Steps}

\crefname{case}{Case}{Cases}

\numberwithin{equation}{section}

\DeclareMathOperator{\Amp}{Amp}

\DeclareMathOperator{\Eff}{Eff}
\DeclareMathOperator{\End}{End}

\DeclareMathOperator{\GL}{GL}

\DeclareMathOperator{\Hom}{Hom}
\DeclareMathOperator{\id}{id}

\DeclareMathOperator{\Mat}{M}

\DeclareMathOperator{\Nef}{Nef}

\DeclareMathOperator{\PL}{PL}
\DeclareMathOperator{\plov}{plov}
\DeclareMathOperator{\pr}{pr}

\DeclareMathOperator{\supp}{supp}
\DeclareMathOperator{\Sym}{Sym}

\newcommand{\isom}{\simeq}

\newcommand{\num}{\equiv}

\newcommand{\bC}{\mathbf{C}}

\newcommand{\bG}{\mathbf{G}}
\newcommand{\bH}{\mathbf{H}}

\newcommand{\bk}{\mathbf{k}}

\newcommand{\bP}{\mathbf{P}}
\newcommand{\bQ}{\mathbf{Q}}
\newcommand{\bR}{\mathbf{R}}

\newcommand{\bZ}{\mathbf{Z}}

\newcommand{\ba}{\boldsymbol{a}}

\newcommand{\bh}{\boldsymbol{h}}
\newcommand{\bl}{\boldsymbol{\ell}}

\newcommand{\bn}{\boldsymbol{n}}
\newcommand{\bq}{\boldsymbol{q}}
\newcommand{\br}{\boldsymbol{r}}

\newcommand{\bt}{\boldsymbol{t}}

\newcommand{\bv}{\boldsymbol{v}}

\newcommand{\bx}{\boldsymbol{x}}

\newcommand{\bone}{\boldsymbol{1}}

\newcommand{\cC}{\mathcal{C}}

\newcommand{\cK}{\mathcal{K}}

\newcommand{\cT}{\mathscr{T}}
\newcommand{\cU}{\mathcal{U}}

\newcommand{\gl}{\mathfrak{gl}}
\newcommand{\stab}{\mathfrak{stab}}

\newcommand{\tot}{\mathrm{tot}}

\newcommand{\sL}{\mathscr{L}}

\newcommand{\sO}{\mathscr{O}}

\newcommand{\Comp}{\mathsf{Comp}}

\newcommand{\N}{\mathsf{N}}

\newcommand{\T}{\mathsf{T}}

\newcommand{\arxiv}[1]{\href{https://arxiv.org/abs/#1}{\texttt{arXiv:#1}}}
\newcommand{\doi}[1]{\href{https://doi.org/#1}{\texttt{doi:#1}}}
\providecommand{\MR}[1]{}
\renewcommand{\MR}[1]{\href{https://mathscinet.ams.org/mathscinet-getitem?mr=#1}{MR#1}}

\newcommand{\CJ}[1]{{\color{blue} J: #1}}
\newcommand{\FH}[1]{{\color{red} H: #1}}

\title[Uniform Tur\'an estimates and sharp bounds]{Uniform Tur\'an estimates and sharp bounds for degree and mixed orbit growth}

\author{Fei Hu}
\address{School of Mathematics, Nanjing University, Nanjing, China}
\email{\href{mailto:fhu@nju.edu.cn}{fhu@nju.edu.cn}}

\author{Chen Jiang}
\address{Shanghai Center for Mathematical Sciences \& School of Mathematical Sciences, Fudan University, Shanghai, China}
\email{\href{mailto:chenjiang@fudan.edu.cn}{chenjiang@fudan.edu.cn}}



\thanks{The first author was supported by NSFC Grant \#12371045.
The second author was supported by the National Key Research and Development Program of China \#2023YFA1010600 and NSFC for Innovative Research Groups \#12121001.}

\date{}

\begin{document}

\begin{abstract}
The degrees of the iterates of a projective endomorphism grow in two layers: an exponential rate measured by the dynamical degrees, and a polynomial correction carried by the peripheral Jordan blocks.
Log-concavity constrains the first layer; we show that the Hodge index theorem already constrains the second, in every codimension and in arbitrary characteristic.

Let $f$ be a surjective endomorphism of a normal projective $d$-fold over an algebraically closed field, with dynamical degrees $\lambda_i$, so that $\textrm{deg}_i(f^n)\asymp\lambda_i^n n^{\mu_i}$ for a unique integer $\mu_i\ge0$.
We prove that these exponents are coupled across three adjacent codimensions: strict log-concavity of $(\lambda_i)_i$ at $i$ forces $\mu_i=0$, while equality gives
\[
0\ \le\ 2\mu_i-\mu_{i-1}-\mu_{i+1}\ \le\ 4 .
\]
Consequently, $\mu_i\le2i(d-i)$, and this is sharp for every $i$.
For a zero-entropy automorphism, we prove that every mixed orbit function indexed by any weak composition $\gamma$ of $d$ is a multivariate quasipolynomial
of even total degree at most $d^2-\sum_j\gamma_j^2$.
For a zero-entropy holomorphic automorphism $g$ of a compact K\"ahler $d$-fold $Y$, the same method gives
\[
\left\|(g^n)^*|_{H^{p,q}(Y,\mathbb{C})}\right\|
=O\bigl(n^{p(d-p)+q(d-q)}\bigr).
\]
Powers of elliptic curves attain all exponent bounds as well as the constant $4$.
\end{abstract}

\maketitle


\section{Introduction}

Log-concavity is the fundamental global constraint on dynamical degrees.
It is therefore natural to ask whether the same Hodge-theoretic structure also controls the peripheral Jordan blocks, and hence the polynomial factors in degree growth.
Optimal dimension-only bounds for these factors were previously known only in codimension one for general zero-entropy automorphisms \cite{DLOZ22,HJ25}.
We answer this question in every codimension for general endomorphisms and, at zero entropy, obtain simultaneous control of all mixed orbit intersections.

Throughout, unless otherwise stated, we work over an algebraically closed field $\bk$ of arbitrary characteristic.
Let $X$ be a normal projective variety of dimension $d$ over $\bk$ and let $f\colon X\to X$ be a surjective endomorphism.
For $0\le i\le d$, let $\N_i(X)_\bR$ denote the finite-dimensional real vector space of numerical classes of $i$-dimensional cycles and set
\[
\N^i(X)_\bR\coloneqq\Hom_\bR\bigl(\N_i(X)_\bR,\bR\bigr).
\]
We call $\N^i(X)_{\bR}$ the space of \textit{numerical dual classes}.
Proper pushforward on $\N_i(X)_\bR$ induces pullback on $\N^i(X)_\bR$.
The space $\N^1(X)_\bR$ agrees with the usual real N\'eron--Severi space, and when $X$ is smooth, $\N^i(X)_\bR$ is canonically isomorphic with the space of codimension-$i$ cycles modulo numerical equivalence.
We refer to \cite{FL17a} for further details.

Fix any norm $\|\cdot\|$ on each $\N^i(X)_\bR$ and any ample Cartier divisor $H$ on $X$.
For $0\le i\le d$ and $n\in\bZ_{\ge0}$, define the $i$-th degree of the iterate $f^n$ by
\[
\deg_i(f^n)\coloneqq (f^n)^*H^i\cdot H^{d-i}.
\]
The standard degree--norm comparison for endomorphisms, stated precisely in \cref{prop:degree-norm-comparison}, shows that the limit
\[
\lambda_i(f) \coloneqq \lim_{n\to +\infty} \bigl(\deg_i(f^n)\bigr)^{1/n}
\]
exists and equals the spectral radius of $f^*|_{\N^i(X)_\bR}$.
Let $\mu_i+1$ be the largest size of Jordan blocks of $f^*|_{\N^i(X)_\bC}$ associated with eigenvalues of modulus $\lambda_i(f)$.
We call $\lambda_i(f)$ the \textit{$i$-th (numerical) dynamical degree} of $f$ and $\mu_i$ the \emph{$i$-th peripheral polynomial-growth exponent} of $f$.


\subsection{Peripheral polynomial growth}

For every $n\ge1$, the Khovanskii--Teissier inequality implies that the degree sequence $(\deg_i(f^n))_{i=0}^d$ is log-concave.
Taking $n$-th roots and passing to the limit, we obtain the log-concavity of the dynamical-degree sequence:
\begin{equation}\label{eq:lambda-log-concave}
\lambda_i(f)^2\ge\lambda_{i-1}(f)\lambda_{i+1}(f)\qquad(1\le i\le d-1).
\end{equation}
Moreover, the Khovanskii--Teissier inequality and the degree--norm comparison, \cref{prop:degree-norm-comparison}, imply that whenever $\lambda_i(f)^2=\lambda_{i-1}(f)\lambda_{i+1}(f)$, one automatically has the concavity of the peripheral polynomial-growth exponents
\[
2\mu_i-\mu_{i-1}-\mu_{i+1}\ge0.
\]

Our new contribution is to bound this discrete Laplacian from above and to determine what happens at a strict break in the dynamical-degree sequence.
Our first main result shows that a strict break at $i$ implies peripheral semisimplicity in codimension $i$, that is, $\mu_i=0$, whereas equality yields a sharp upper bound for the discrete Laplacian.

\begin{definition}[Discrete Laplacian and maximal affine interval]
\label{def:maximal-affine-interval}
For a real sequence $(a_0,\dots,a_d)$, define its \textit{discrete Laplacian} by
\[
\Delta a_i \coloneqq 2a_i-a_{i-1}-a_{i+1}, \qquad 1\le i\le d-1.
\]
We intentionally suppress the negative sign by convention.
Note that $(a_i)_i$ is a concave sequence if and only if $\Delta a_i\ge0$ for all $1\le i\le d-1$.

An index interval $[r,s]$, with $0\le r<s\le d$, is \emph{affine} for $(a_i)_i$ if $\Delta a_i=0$ for every $r<i<s$,
and is \emph{maximal affine} if it is maximal under inclusion among affine index intervals.

\end{definition}

\begin{theorem}[Peripheral polynomial growth]
\label{thm:endomorphism-peripheral-growth}
Let $X$ be a normal projective variety of dimension $d\ge2$ over $\bk$,
and let $f\colon X\to X$ be a surjective endomorphism.
For every $1\le i\le d-1$, let $\mu_i$ be the $i$-th peripheral polynomial-growth exponent of $f$.
Then the following hold.
\begin{enumerate}[label=\textnormal{(\arabic*)}]
\item If $\lambda_i(f)^2>\lambda_{i-1}(f)\lambda_{i+1}(f)$, then $\mu_i=0$.
Equivalently, every Jordan block of $f^*|_{\N^i(X)_\bC}$ associated with an eigenvalue of modulus $\lambda_i(f)$ has size one.

\item If $\lambda_i(f)^2=\lambda_{i-1}(f)\lambda_{i+1}(f)$, then
\begin{equation}\label{eq:intro-discrete-Laplacian}
0\le \Delta \mu_i = 2\mu_i-\mu_{i-1}-\mu_{i+1}\le4.
\end{equation}
\end{enumerate}
Consequently, if $[r,s]$ is a maximal affine interval of $(\log\lambda_i(f))_i$, then
\[
\mu_i\le2(i-r)(s-i)
\qquad(r\le i\le s).
\]
In particular,
\[
\mu_i\le2i(d-i)
\qquad(0\le i\le d).
\]
\end{theorem}

To our knowledge, the upper bound in \eqref{eq:intro-discrete-Laplacian} is the first universal estimate coupling the peripheral polynomial-growth exponents in three adjacent codimensions.
The resulting discrete Laplacian constraint is the basic new phenomenon.

We say that $f$ has \emph{zero entropy} if $\lambda_1(f)=1$, equivalently, $\lambda_i(f)=1$ for all $i$ (see, e.g., \cite[Lemma~2.2]{Hu-GK-AV}).
In particular, such an $f$ is necessarily an automorphism.
Since $f^*$ acts invertibly on the integral N\'eron--Severi lattice modulo torsion, this is equivalent to quasi-unipotence of $f^*|_{\N^1(X)_\bR}$ by Kronecker's theorem.
When $X$ is smooth over $\bC$, it is also equivalent to vanishing topological entropy by the Gromov--Yomdin theorem \cite{Gromov03,Yomdin87}.

\begin{corollary}
\label{cor:zero-entropy-degree-growth}
Under the hypotheses of \cref{thm:endomorphism-peripheral-growth}, assume further that $f$ has zero entropy.
Then $\mu_i=\mu_{d-i}\in2\bZ_{\ge0}$ for every $0\le i\le d$ and, for every $1\le i\le d-1$,
\[
\Delta \mu_i = 2\mu_i-\mu_{i-1}-\mu_{i+1}\in\{0,2,4\}.
\]
Moreover,
\[
\deg_i(f^n)
\asymp
\left\|(f^n)^*|_{\N^i(X)_\bR}\right\|
\asymp
n^{\mu_i}
=O(n^{2i(d-i)})
\qquad (n\to\infty).
\]
\end{corollary}

\begin{remark}
At a strict break, i.e., the inequality \eqref{eq:lambda-log-concave} is strict, \cref{thm:endomorphism-peripheral-growth} gives semisimplicity only on the peripheral spectrum: it implies neither that $\lambda_i(f)$ is simple nor that the entire action on $\N^i(X)_\bC$ is semisimple, although Jordan blocks at smaller-modulus eigenvalues affect only exponentially lower-order terms.
Dang--Favre \cite[Theorem~1.1]{DF21} proved a stronger codimension-one asymptotic for dominant rational maps in characteristic zero, and Favre--Wulcan \cite[Theorem~D]{FW12} proved a sharper arbitrary-codimension asymptotic for rational monomial maps; our result instead treats every codimension for projective endomorphisms in arbitrary characteristic and gives the sharp discrete-Laplacian bound and the resulting affine-interval bound.
All three values $0,2,4$ in \cref{cor:zero-entropy-degree-growth} can occur simultaneously.
\end{remark}

The same proof applies to surjective holomorphic endomorphisms of compact K\"ahler manifolds after replacing $\N^i(X)_\bR$ by $H^{i,i}(Y,\bR)$ and an ample class by a K\"ahler class: the required mass--norm comparison holds, while pullback preserves the K\"ahler cone and acts as a similitude of the intersection form.
For zero-entropy automorphisms, combining the resulting diagonal estimates with \cite{Dinh05} gives the following Hodge-bidegree bound.

\begin{corollary}
\label{cor:kahler-growth}
Let $Y$ be a compact K\"ahler manifold of dimension $d\ge2$, and let $g\colon Y\to Y$ be a holomorphic automorphism of zero topological entropy.
Then for all $0\le p,q\le d$,
\begin{equation}\label{eq:intro-kahler-pq}
\left\|(g^n)^*|_{H^{p,q}(Y,\bC)}\right\|
=O\bigl(n^{p(d-p)+q(d-q)}\bigr)\qquad (n\to\infty).
\end{equation}
Equivalently, every Jordan block of $g^*|_{H^{p,q}(Y,\bC)}$ has size at most $p(d-p)+q(d-q)+1$.
\end{corollary}

\begin{remark}
Polynomial degree and norm growth for zero-entropy automorphisms was known for surfaces (see, e.g., \cite[Lemma~5.4]{AVdB90} or \cite[Appendix]{DF01}).
It was revived by Lo Bianco \cite{LB19}, who proved $\mu_1\le4$ in dimension three.
Dinh--Lin--Oguiso--Zhang \cite[Theorem~1.1]{DLOZ22} proved the optimal bound $\mu_1\le2d-2$ in the K\"ahler setting and
obtained the general estimate $O\bigl(n^{(p'+q')(d-1)}\bigr)$ on $H^{p,q}$,
where $p'=\min\{p,d-p\}$ and $q'=\min\{q,d-q\}$;
the first-degree bound was later extended to arbitrary characteristic in \cite[Corollary~1.6]{HJ25}.
The exponent in \eqref{eq:intro-kahler-pq} improves their Hodge-bidegree exponent by $p'(p'-1)+q'(q'-1)$ and is sharp for every $(p,q)$ \cite[Remark~4.1]{DLOZ22}; see also \cref{thm:sharpness-main}\ref{thm:abelian-hodge-sharpness}.
Our result also recovers \cite[Theorem~1.1]{DLOZ22} and \cite[Corollary~1.6]{HJ25} by a different method; see \cref{rmk:specializations}.
\end{remark}

\subsection{Zero-entropy mixed orbit growth}\label{sec:1.2}

Using our framework of dynamical Hodge--Lorentz packages, we next study mixed orbit functions for zero-entropy automorphisms associated with arbitrary weak compositions.

Denote the finite set of weak compositions of $d$ into $d$ ordered parts by
\[
\Comp(d)
\coloneqq
\left\{
\gamma=(\gamma_1,\ldots,\gamma_d)\in\bZ_{\ge0}^d:
|\gamma|\coloneqq\sum_{j=1}^d\gamma_j=d
\right\}.
\]
For $\gamma\in\Comp(d)$, a tuple $\bH=(H_1,\ldots,H_d)$ of ample Cartier divisors on $X$, and $\bn=(n_1,\ldots,n_d)\in\bZ_{\ge0}^d$, the main object is the following \textit{mixed orbit function}
\[
\Phi_{f,\gamma,\bH}(\bn)
\coloneqq
\prod_{j=1}^d
\bigl((f^{n_j})^*H_j\bigr)^{\gamma_j},
\]
where the product of $d$ divisors denotes the top intersection number.
If $f$ is an automorphism, then $\Phi_{f,\gamma,\bH}$ naturally extends to $\bZ^d$.

A function on $\bZ^d$ is a \emph{multivariate quasipolynomial} if it agrees with a polynomial on each residue class modulo some positive integer; its total degree is the largest degree of these constituent polynomials.

\begin{theorem}[Zero-entropy mixed orbit growth]
\label{thm:zero-entropy-mixed}
Let $X$ be a normal projective variety of dimension $d\ge2$ over $\bk$, and let $f\colon X\to X$ be a zero-entropy automorphism.
For every $\gamma\in\Comp(d)$ and every tuple $\bH$ of ample Cartier divisors, $\Phi_{f,\gamma,\bH}$ is a multivariate quasipolynomial whose period may be chosen to depend only on $f$.
Every constituent polynomial has even total degree, and the total degree of the
quasipolynomial satisfies
\begin{equation}\label{eq:intro-zero-entropy-mixed-bound}
\deg_{\tot}\Phi_{f,\gamma,\bH}
\le
b_\gamma\coloneqq
 d^2-\|\gamma\|_2^2
 =
 2\sum_{1\le r<s\le d}\gamma_r\gamma_s.
\end{equation}
In particular, $\Phi_{f,\gamma,\bH}(\bn)=O\bigl(\|\bn\|_\infty^{b_\gamma}\bigr)$, and for every $\bl\in\bZ^d$,
\begin{equation}\label{eq:intro-zero-entropy-ray}
\Phi_{f,\gamma,\bH}(n\bl)=O(n^{b_\gamma})
\qquad(n\to\infty).
\end{equation}
\end{theorem}

\begin{remark}
The mixed orbit functions in \cref{thm:zero-entropy-mixed} are related to Xie's mixed degrees \cite[\S3.2]{Xie-SC}.
For dominant rational self-maps, Xie obtained two-sided comparisons with ordinary degree sequences and determined their exponential growth up to subexponential factors \cite[Proposition~3.3 and Corollary~3.4]{Xie-SC}.
At zero entropy, these estimates do not detect the remaining polynomial correction.
In the functorial setting, our \cref{thm:zero-entropy-mixed} gives multivariate quasipolynomiality and the sharp total-degree bound \eqref{eq:intro-zero-entropy-mixed-bound}.
\end{remark}

\begin{remark}
\label{rmk:specializations}
The two basic specializations of \cref{thm:zero-entropy-mixed} at once recover degree growth and polynomial volume growth.
Taking
\[
\gamma=(i,d-i,0,\ldots,0),\qquad
\bn=(n,0,\ldots,0),\qquad
\bH=(H,\ldots,H),
\]
gives
\[
\deg_i(f^n)=\Phi_{f,\gamma,\bH}(\bn)
=O\bigl(n^{2i(d-i)}\bigr).
\]
The case when $i=1$ was proved in \cite[Corollary~1.6]{HJ25}.
On the other hand, recall that
\[
\plov(f)
\coloneqq
\limsup_{n\to+\infty}
\frac{\log\bigl((\sum_{m=0}^{n-1}(f^m)^*H)^d\bigr)}
     {\log n}.
\]
Expanding the intersection product and applying the theorem with $\gamma=(1,\ldots,1)$ gives $n^d$ terms of order $O(n^{d^2-d})$, and hence $\plov(f)\le d^2$, recovering \cite[Corollary~1.2]{HJ25}; the refined estimate $\plov(f)\le(\mu_1/2+1)d$ was obtained in \cite[Theorem~1.1]{HJ25}.
Polynomial volume growth was introduced in \cite{CPR21}, and its relation to Gelfand--Kirillov dimension was observed in \cite{LOZ25}, building on \cite{ATVdB90,AVdB90,Keeler00}; the dynamical intersection polynomials introduced in our recent work \cite{HJ-plov-lower} to prove the sharp lower bound for $\plov$ are prototypes of the present mixed orbit functions.
Thus \eqref{eq:intro-zero-entropy-mixed-bound} places these one-parameter invariants inside one sharp bound for arbitrary weak compositions.
\end{remark}

\subsection{Proof strategy}
\Cref{sec:Hodge-Lorentz} introduces dynamical Hodge--Lorentz packages $(V,\cC,Q,T)$, establishes the multiplicative Jordan--Chevalley decomposition, and proves the Lorentzian Ptolemy inequality obtained by freezing $d-2$ arguments of $Q$.
In \cref{sec:uniform-Turan}, writing $T=Se^N$ and perturbing the initial cone vectors in the $N$-direction before iteration makes the Ptolemy constant independent of the orbit times and frozen arguments.
Expanding the resulting inequality gives the orbit-uniform mixed Tur\'an upper estimate and the degree-detection criterion.

\Cref{sec:abstract-growth} begins with the unipotent two-block case and then treats general mixed unipotent orbit functions.
At zero entropy, a finite iterate of the relevant pullback action is unipotent.
Diagonal translation invariance and cone positivity give parity of the total degree; an Euler--Tur\'an identity selects an active transfer pair, and the orbit-uniform estimate supplies the local defect bound $4$.
A maximum principle on the transfer graph then compares the degree function with $b_\gamma=d^2-\|\gamma\|_2^2$,
and treating residue classes separately gives quasipolynomiality for the original zero-entropy automorphism.

For the general two-block sequence
\[
D_i(n)=Q\bigl((T^nh_1)^{[i]},h_2^{[d-i]}\bigr),
\]
a shifted exponential--polynomial expansion turns $D_i(n)\asymp\lambda_i^n n^{\mu_i}$ into a lower bound for the same Tur\'an expression.
Together with the upper bound, this forces $\mu_i=0$ at a strict log-concavity break and gives $0\le2\mu_i-\mu_{i-1}-\mu_{i+1}\le4$ in the equality case; the discrete comparison principle then yields the quadratic bound on every maximal affine interval.

Finally, \cref{sec:applications} supplies the projective degree--norm and K\"ahler mass--norm comparisons and specializes the abstract results to the geometric models of \cref{ex:geometric-HL-models}.
Appendix~\ref{sec:abelian-sharpness} uses Cauchy--Binet and confluent Vandermonde determinants to construct powers of elliptic curves simultaneously attaining the mixed, degree-growth, affine-interval, and Hodge-bidegree bounds.

\subsection*{Acknowledgments}
The first author gratefully acknowledges the hospitality of the Fields Institute during Summer 2026, when this work was carried out.
C.~Jiang is a member of the Key Laboratory of Mathematics for Nonlinear Sciences, Fudan University.
The authors also thank Lie Fu for stimulating discussions on mixed-volume growth in 2023.

\subsection*{Declaration on the use of generative AI}
The initial manuscript was assembled from two author-written drafts with the assistance of OpenAI's ChatGPT 5.5. The authors then substantially rewrote and revised the manuscript.
At a late stage, ChatGPT was extensively used in revising the language and exposition throughout the manuscript.
It aslo assisted in cross-checking proofs, particularly the matrix calculations in the Appendix. All arguments and computations were independently derived and verified by the authors, who take full responsibility for the content of the manuscript.

\section{Dynamical Hodge--Lorentz packages and Ptolemy inequalities}\label{sec:Hodge-Lorentz}


\subsection{Infinitesimal symmetries and Jordan--Chevalley decomposition}
\begin{definition}\label{def:stabQ}
Let $V$ be a finite-dimensional real vector space and let $Q\in\Sym^d(V^\vee)$ be a non-zero symmetric $d$-linear form.
\begin{enumerate}
\item An automorphism $T\in \GL(V)$ is called a \textit{$Q$-similitude} if \[Q(Tv_1,\ldots,Tv_d)
=
\chi(T)Q(v_1,\ldots,v_d)\]
for some $\chi(T)\in \bR^*$ and all $v_1,\ldots,v_d\in V$. If $\chi(T)=1$, we call $T$ a \emph{$Q$-isometry}. 
    \item The \textit{infinitesimal stabilizer} of $Q$ is defined by
\[
\stab(Q) \coloneqq \left\{ N\in\gl(V): \sum_{i=1}^dQ(v_1,\ldots,Nv_i,\ldots,v_d)=0 \text{ for all }v_1,\ldots,v_d\in V \right\}.
\]
\end{enumerate}

\end{definition}

\begin{lemma}[Diagonal $N$-invariance]\label{lem:stabQ}
Let $V$ be a finite-dimensional real vector space and let $Q\in\Sym^d(V^\vee)$ be a non-zero symmetric $d$-linear form.
For $N\in\gl(V)$, one has $N\in\stab(Q)$ if and only if $e^{tN}$ is a $Q$-isometry for every $t\in\bR$, that is, 
\begin{equation}\label{eq:etN-preserves-Q}
Q(e^{tN}v_1,\ldots,e^{tN}v_d)=Q(v_1,\ldots,v_d)\, \text{ for every }  t\in\bR  \text{ and } v_1,\ldots,v_d\in V.
\end{equation}
If moreover $N$ is nilpotent, it suffices to verify \eqref{eq:etN-preserves-Q} at $t=1$:
\begin{equation}\label{eq:eN-preserves-Q}
Q(e^Nv_1,\ldots,e^Nv_d)=Q(v_1,\ldots,v_d)\, \text{ for all } v_1,\ldots,v_d\in V.
\end{equation}
\end{lemma}

\begin{proof}
For fixed $v_1,\ldots,v_d$, define a real-valued function
\[
g(t)\coloneqq Q(e^{tN}v_1,\ldots,e^{tN}v_d), \qquad t\in \bR.
\]
Multilinearity of $Q$ gives
\[
g'(t)=\sum_{i=1}^d
Q(e^{tN}v_1,\ldots,Ne^{tN}v_i,\ldots,e^{tN}v_d).
\]
Thus $N\in\stab(Q)$ implies $g'\equiv0$.
Hence $g\equiv g(0)$ which is exactly \eqref{eq:etN-preserves-Q}.
The converse follows by differentiating the equation \eqref{eq:etN-preserves-Q} at $t=0$, which gives $g'(0)=0$.

If $N$ is nilpotent, then the above function $g$ is a polynomial.
If \eqref{eq:eN-preserves-Q} holds, iteration gives $g(n)=g(0)$ for every $n\ge1$; hence $g-g(0)$ has infinitely many zeros and is identically zero.
The reverse implication is immediate.
\end{proof}

For an arbitrary $Q$-similitude $T\in \GL(V)$, we use the multiplicative Jordan--Chevalley decomposition to extract its unipotent part.

\begin{lemma}
\label{lem:multi-JC}
Let $V$ be a finite-dimensional real vector space, let $0\ne Q\in\Sym^d(V^\vee)$, and let $T\in\GL(V)$ be a $Q$-similitude. 
Write
\[
T=Se^N=e^NS
\]
for the multiplicative Jordan--Chevalley decomposition of $T$. Then $S$ and $e^N$ are  $Q$-similitudes and $N\in\stab(Q)$.
\end{lemma}

\begin{proof}
Denote by 
$\operatorname{GSim}(Q)$ the set of all $Q$-similitudes. 
Then it is the stabilizer of $[Q]$ for the natural $\GL(V)$-action on $\bP(\Sym^d(V^\vee))$, and hence is an algebraic subgroup.
Its action on the invariant line $\bR Q$ defines an algebraic character
\[
\chi\colon\operatorname{GSim}(Q)\longrightarrow\bG_m.
\]
Applying the Jordan decomposition theorem for algebraic groups 
(see \cite[Theorem~4.4(1), (4)]{Borel91} or \cite[Theorem~15.3(a), (c)]{GTM21})  
to $T=Se^N\in \operatorname{GSim}(Q)$ gives
\[
S,e^N\in\operatorname{GSim}(Q)
\]
and identifies $\chi(e^N)$ with the unipotent part of $\chi(T)\in \bG_m$, which is just $1$.
Thus $e^N$ is a $Q$-isometry, and \cref{lem:stabQ} gives $N\in\stab(Q)$.
\end{proof}

\subsection{Static and dynamical Hodge--Lorentz packages}


\begin{definition}\label{def:HL-package}
A \emph{Hodge--Lorentz package of degree $d\ge2$} is a triple $(V,\cC,Q)$ consisting of a finite-dimensional real vector space $V$, a nonempty open convex cone $\cC\subset V$, and a symmetric $d$-linear form $Q\in\Sym^d(V^\vee)$ satisfying the following conditions.
\begin{enumerate}[label=\textnormal{(\roman*)}, ref=\textnormal{(\roman*)}]
\item Positivity: For every $h_1,\ldots,h_d\in\cC$,
\[
Q(h_1,\ldots,h_d)>0.
\]
\item Lorentzian: For every $h_1,\ldots,h_{d-2}\in\cC$, the symmetric bilinear form
\[
Q(-,-,h_1,\ldots,h_{d-2})
\colon
V\times V\longrightarrow\bR
\]
has positive index at most one.
\end{enumerate}
\end{definition}

Condition \textnormal{(i)} supplies a positive direction, so each contracted form in \textnormal{(ii)} has positive index exactly one; zero eigenvalues, and hence degeneracy, are allowed.

The associated homogeneous polynomial is
\[
F_Q(v)
\coloneqq
\frac1{d!}Q(v,\ldots,v).
\]
After polarization, the preceding definition is the multilinear formulation of the cone-Lorentzian condition for $F_Q$ in the sense of Br\"anden--Leake \cite{BL26}; on the positive orthant it recovers the corresponding index-one formulation in the Lorentzian-polynomial theory of Br\"anden--Huh \cite{BH20}.
We do not require $F_Q$ to be the volume polynomial of a geometric object, although the standard examples arise from intersection forms, mixed volumes, and mixed discriminants.


\begin{definition}\label{def:dynamical-HL-package}
A \emph{dynamical Hodge--Lorentz package of degree $d\ge2$} is a quadruple $(V,\cC,Q,T)$, where $(V,\cC,Q)$ is a Hodge--Lorentz package of degree $d$ and $T\in\GL(V)$ is  a $Q$-similitude satisfying 
 $T(\cC)\subseteq\cC$. 
Such $T$ is called a \emph{cone-preserving $Q$-similitude}.
If $T=e^N$ is unipotent, we call the package \emph{unipotent}.
\end{definition}

\begin{example}\label{ex:geometric-HL-models}
The two principal geometric models in application are the following. The corresponding geometric applications will be developed in \Cref{sec:applications}. 
\begin{enumerate}[label=\textnormal{(\roman*)}]
\item Let $X$ be a normal projective variety of dimension $d\ge2$ over $\bk$.
Denote by $\Amp(X)\subset\N^1(X)_\bR$ the ample cone of $X$.
Then
\[
V=\N^1(X)_\bR,
\qquad
\cC=\Amp(X),
\qquad
Q(D_1,\ldots,D_d)=D_1\cdots D_d
\]
is a Hodge--Lorentz package by the mixed Hodge index theorem (see, e.g., \cite[Proposition~2.9]{Hu20a}).
For a surjective endomorphism $f$ of $X$, $f$ is finite and by the projection formula, $T=f^*\in \GL(V)$ is a cone-preserving $Q$-similitude with $\chi(T)=\deg(f)$. So $(V,\cC,Q,T)$ is a dynamical Hodge--Lorentz package of degree $d$.

\item Let $Y$ be a compact K\"ahler manifold of dimension $d\ge2$.
Denote by $\cK_Y\subset H^{1,1}(Y,\bR)$  the K\"ahler cone of $Y$.
Then
\[
V=H^{1,1}(Y,\bR),
\qquad
\cC=\cK_Y,
\qquad
Q(\alpha_1,\ldots,\alpha_d)
=
\int_Y\alpha_1\wedge\cdots\wedge\alpha_d
\]
is a Hodge--Lorentz package by the mixed Hodge--Riemann relations \cite{Gromov90,DN06}.
For a surjective holomorphic endomorphism $g$ of $Y$, $g$ is also finite and by the numerical characterization of K\"ahler cone \cite{DP04}, $T=g^*\in \GL(V)$ is a cone-preserving $Q$-similitude. So $(V,\cC,Q,T)$ is a  dynamical Hodge--Lorentz package of degree $d$.
\end{enumerate}
\end{example}

\subsection{Mixed coefficients, mixed orbit functions, and unipotent-shift polynomials}

Recall that $\Comp(d)$ denotes the set of weak compositions of $d$ into $d$ ordered parts, as defined in \cref{sec:1.2}.

\begin{definition}\label{def:mixed-Q-coefficients}
Let $(V,\cC,Q)$ be a Hodge--Lorentz package of degree $d\ge2$,
let $\gamma\in\Comp(d)$, and let $\bv=(v_1,\ldots,v_d)\in V^d$.
Define the \textit{mixed $Q$-coefficient} by
\[
Q_\gamma(\bv)\coloneqq Q(v_1^{[\gamma_1]},\ldots,v_d^{[\gamma_d]}),
\]
where $v_j^{[\gamma_j]}$ means that $v_j$ appears $\gamma_j$ times with $v_j^{[0]}$ omitted.
Equivalently, polarization gives
\[
F_Q\Biggl(\sum_{i=1}^d x_iv_i\Biggr)
=
\sum_{\gamma\in\Comp(d)}
Q_\gamma(\bv)\frac{x^{\gamma}}{\gamma!},
\]
where $\gamma!\coloneqq \gamma_1!\cdots\gamma_d!$ and $x^{\gamma}\coloneqq x_1^{\gamma_1}\cdots x_d^{\gamma_d}$.
Thus $Q_\gamma(\bv)$ is the coefficient of $x^\gamma/\gamma!$ in the expansion of $F_Q$.
\end{definition}


From here until the end of this subsection, let $(V,\cC,Q,T)$ be a dynamical Hodge--Lorentz package of degree $d\ge2$, let $\gamma\in\Comp(d)$, and let $\bh=(h_1,\ldots,h_d)\in\cC^d$.

\begin{definition}\label{def:mixed-orbit-functions}
For $\bn=(n_1,\ldots,n_d)\in\bZ^d$, set
\[
T^{\bn}\bh
\coloneqq
(T^{n_1}h_1,\ldots,T^{n_d}h_d)
\]
and define the \emph{mixed orbit function}
\[
\Phi_{T,\gamma,\bh}(\bn)
\coloneqq
Q_\gamma(T^{\bn}\bh)
=
Q_\gamma(T^{n_1}h_1,\ldots,T^{n_d}h_d).
\]

If the package is unipotent, i.e., $T=e^N$, then the mixed orbit function has the \textit{canonical polynomial extension}
\[
\Phi_{T,\gamma,\bh}(\bt)
\coloneqq
Q_\gamma\bigl(
e^{t_1N}h_1,\ldots,e^{t_dN}h_d
\bigr),
\qquad
\bt\in\bR^d.
\]
\end{definition}

For a general dynamical Hodge--Lorentz package that may not be unipotent, we introduce auxiliary unipotent-shift polynomials used to obtain the orbit-uniform Tur\'an upper estimate \cref{thm:turan-upper}.
Write
$T=Se^N=e^NS$
for the multiplicative Jordan--Chevalley decomposition; see \cref{lem:multi-JC}.

\begin{definition}\label{def:unipotent-shift-poly}
For $\bn=(n_1,\ldots,n_d)\in\bZ^d$ and $\bt=(t_1,\ldots,t_d)\in\bR^d$, define the \emph{unipotent-shift polynomial} of $\Phi_{T,\gamma,\bh}$ at $\bn$ by
\begin{align*}
\phi_{\gamma,\bn}(\bt)
&\coloneqq
Q_\gamma\bigl(
e^{t_1N}S^{n_1}h_1,\ldots,
e^{t_dN}S^{n_d}h_d
\bigr)=Q_\gamma\bigl(
e^{(t_1-n_1)N}T^{n_1}h_1,\ldots,
e^{(t_d-n_d)N}T^{n_d}h_d
\bigr).
\end{align*}
Thus $\phi_{\gamma,\bn}$ is a polynomial in $\bt$, even though $T$ may not be unipotent, and
$\phi_{\gamma,\bn}(\bn)=\Phi_{T,\gamma,\bh}(\bn)$.
If $T=e^N$ is unipotent, then for every $\bt\in\bR^d$, one has
$\phi_{\gamma,\bn}(\bt)=\Phi_{T,\gamma,\bh}(\bt)$,
independently of $\bn$.

When $\gamma=(i,d-i,0,\ldots,0)$ has only two nontrivial blocks and $\bn=(n,0,\ldots,0)$, by \cref{lem:stabQ}, the corresponding unipotent-shift polynomial $\phi_{\gamma,\bn}$ at $\bn$ is
\[
Q\bigl((e^{t_1N}S^nh_1)^{[i]},(e^{t_2N}h_2)^{[d-i]}\bigr)=Q\bigl((e^{(t_1-t_2)N}S^nh_1)^{[i]},h_2^{[d-i]}\bigr).
\]
Define the one-variable unipotent-shift polynomial at $n$ by
\[
\phi_{i,n}(t)
\coloneqq 
Q\bigl((e^{tN}S^nh_1)^{[i]},h_2^{[d-i]}\bigr)
=Q\bigl((e^{(t-n)N}T^nh_1)^{[i]},h_2^{[d-i]}\bigr).
\]
If $T=e^N$ is unipotent, this polynomial is independent of $n$. 
\end{definition}

\subsection{Lorentzian reverse Cauchy--Schwarz and Ptolemy inequality}

The basic geometric input is a four-point Ptolemy inequality for Lorentzian bilinear forms.
The exact zero-diagonal matrix characterization is \cite[Proposition~3.1]{HHMWW}.
See also Huh's lectures \cite{Huh-SRI} at the 2025 Summer Research Institute.
For the reader's convenience, we include a short linear-algebraic proof.

\begin{theorem}\label{thm:lorentzian-ptolemy} 
Let $V$ be a finite-dimensional real vector space and let $B\colon V\times V\to \bR$ be a symmetric bilinear form with positive index at most $1$.
\begin{enumerate}[label=\textnormal{(\arabic*)}]
\item {\it Reverse Cauchy--Schwarz inequality}: For every $v_1, v_2\in V$ with $B(v_1, v_1)\geq 0$, we have
\begin{equation}\label{eq:RCS-KT}
B(v_1, v_2)^2\geq B(v_1, v_1)B(v_2, v_2).
\end{equation}

\item {\it Four-point Ptolemy inequality}: If $v_1, \dots, v_4\in V$ satisfy $B(v_i, v_j)\geq 0$ for all $1\leq i, j\leq 4$, then
\begin{equation}\label{eq:Ptolemy}
\left|\sqrt{B(v_1,v_2)B(v_3,v_4)}-\sqrt{B(v_1,v_4)B(v_2,v_3)}\right|
\le
\sqrt{B(v_1,v_3)B(v_2,v_4)}.
\end{equation}
\end{enumerate}
\end{theorem}

\begin{proof}
(1) We may assume that $B(v_1, v_1)>0$, since otherwise the inequality is clear.
The Gram matrix
\[
\begin{pmatrix}
B(v_1,v_1) & B(v_1,v_2)\\
B(v_1,v_2) & B(v_2,v_2)
\end{pmatrix}
\]
of $v_1,v_2$ has exactly one positive eigenvalue by \cite[Corollary~4.5.11]{HJ13} and hence nonpositive determinant.
This proves the reverse Cauchy--Schwarz inequality \eqref{eq:RCS-KT}.

(2) Write $b_{ij}\coloneqq B(v_i,v_j)\ge0$.
The Gram matrix $(b_{ij})_{1\le i, j\le 4}$ has at most one positive eigenvalue by \cite[Corollary~4.5.11]{HJ13}.
Subtracting a positive-semidefinite matrix cannot increase the positive index.
Therefore, after subtracting the nonnegative main diagonal, the matrix
\[
M=\begin{pmatrix}
0 & b_{12} & b_{13} & b_{14}\\
b_{12} & 0 & b_{23} & b_{24}\\
b_{13} & b_{23} & 0 & b_{34}\\
b_{14} & b_{24} & b_{34} & 0
\end{pmatrix}
\]
has at most one positive eigenvalue.
If $M$ is singular, then $\det(M)=0$.
If $M$ is nonsingular, then its zero trace and positive index at most one force its signature to be $(1,3)$.
Thus $\det(M)\leq0$ in all cases.
Set
\[
a=\sqrt{b_{12}b_{34}},
\qquad
b=\sqrt{b_{13}b_{24}},
\qquad
c=\sqrt{b_{14}b_{23}}.
\]
A direct computation gives
\[
\det M = \bigl(b^2-(a+c)^2\bigr)\bigl(b^2-(a-c)^2\bigr).
\]
Since $0\le (a-c)^2\le (a+c)^2$ and $b\ge0$, the inequality $\det M\le0$ is equivalent to
\[
(a-c)^2\le b^2\le(a+c)^2,
\]
or equivalently, $|a-c|\le b\le a+c$.
The first inequality is exactly \eqref{eq:Ptolemy}.
\end{proof}

\begin{remark}
The reverse Cauchy--Schwarz inequality \eqref{eq:RCS-KT} is also called the Khovanskii--Teissier inequality.
There is also a reverse Khovanskii--Teissier inequality on three variables first observed by Xiao \cite[Remark~3.1]{Xiao15}, though we do not need it in this paper:
if $v_1, v_2, v_3\in V$ satisfy $B(v_i, v_j)\ge 0$ for all $1\le i, j\le 3$, then
\begin{equation*}
B(v_1,v_1)B(v_2,v_3) \le 2B(v_1,v_2)B(v_1,v_3).
\end{equation*}
See \cite{LX16,LX17,Dang20,JL23} for its general mixed versions.
\end{remark}

\begin{corollary}[Log-concavity of mixed $Q$-coefficients]\label{cor:mixed-log-concavity}
Let $(V,\cC,Q)$ be a Hodge--Lorentz package of degree $d\ge2$ as in \cref{def:HL-package}.
For $h_1,h_2,h_3,\ldots,h_d\in\cC$ one has
\[
Q(h_1,h_2,h_3,\ldots,h_d)^2\ge Q(h_1,h_1,h_3,\ldots,h_d)Q(h_2,h_2,h_3,\ldots,h_d).
\]
In particular, for the two-block mixed $Q$-coefficients, we have for every $1\le i\le d-1$,
\[
Q\bigl(h_1^{[i]},h_2^{[d-i]}\bigr)^2 \ge Q\bigl(h_1^{[i-1]},h_2^{[d-i+1]}\bigr)Q\bigl(h_1^{[i+1]},h_2^{[d-i-1]}\bigr).
\]
\end{corollary}
\begin{proof}
It follows from \cref{thm:lorentzian-ptolemy}(1).
\end{proof}

\begin{corollary}[Ptolemy inequality for mixed $Q$-coefficients]\label{cor:mixed-ptolemy}
Let $(V,\cC,Q)$ be a Hodge--Lorentz package of degree $d\ge2$ as in \cref{def:HL-package}.
Let $h_1,\ldots,h_4,\omega_3,\ldots,\omega_{d}\in\cC$, and define the symmetric bilinear form
\[
B\colon V\times V\longrightarrow\bR,
\qquad
B(u,v)\coloneqq Q(u,v,\omega_3,\ldots,\omega_{d}).
\]
Then
\[
\Bigl|\sqrt{B(h_1,h_2)B(h_3,h_4)}-\sqrt{B(h_1,h_4)B(h_2,h_3)}\Bigr|
\le
\sqrt{B(h_1,h_3)B(h_2,h_4)}.
\]
\end{corollary}

\begin{proof}
It follows from \cref{thm:lorentzian-ptolemy}(2).
\end{proof}

\section{Orbit-uniform Tur\'an upper estimates}\label{sec:uniform-Turan}

Throughout this section, let $(V,\cC,Q,T)$ be a dynamical Hodge--Lorentz package of degree $d\ge2$ as in \cref{def:dynamical-HL-package}.
Write $T=Se^N=e^NS$ for the multiplicative Jordan--Chevalley decomposition of $T$; see \cref{lem:multi-JC}.


\subsection{Uniform infinitesimal Ptolemy inequality}

We first prove a uniform infinitesimal form of the four-point Ptolemy inequality.

\begin{lemma}\label{lem:uniform-infinitesimal-ptolemy}
Fix $\omega_1,\omega_2\in\cC$.
Then there exists a constant $C>0$, depending only on the package and on $\omega_1,\omega_2$, such that the following holds.

For every $m_1,m_2\in\bZ_{\ge0}$ and every $\omega_3,\ldots,\omega_d\in\cC$, the symmetric bilinear form
\[
B\colon V\times V\longrightarrow\bR,
\qquad
B(u,v)\coloneqq Q(u,v,\omega_3,\ldots,\omega_d)
\]
satisfies
\begin{align}\label{eq:uniform-infinitesimal-ptolemy}
&\Bigl|
 B(T^{m_1}\omega_1,T^{m_2}\omega_2)
 B(NT^{m_1}\omega_1,NT^{m_2}\omega_2)
 -B(NT^{m_1}\omega_1,T^{m_2}\omega_2)
  B(T^{m_1}\omega_1,NT^{m_2}\omega_2)
\Bigr|\notag\\
&\qquad\le
C B(T^{m_1}\omega_1,T^{m_2}\omega_2)
\sqrt{
 B(T^{m_1}\omega_1,T^{m_1}\omega_1)
 B(T^{m_2}\omega_2,T^{m_2}\omega_2)
}.
\end{align}
\end{lemma}

\begin{proof}
Choose $\epsilon>0$ sufficiently small such that, for $i=1,2$,
\[
\omega_{i,\epsilon}\coloneqq \omega_i+\epsilon N\omega_i \in \cC \qquad \text{and} \qquad \omega_i-\epsilon \omega_{i,\epsilon}\in\cC.
\]
Such a choice is possible because $\cC$ is open and both displayed vectors converge to $\omega_i$ as $\epsilon\to0^+$. We will show that $C=2\epsilon^{-4}$ satisfies the requirement, which depends only on the package and on $\omega_1,\omega_2$. 

Fix $m_1,m_2\in\bZ_{\ge0}$ and $\omega_3,\ldots,\omega_d\in\cC$.
Set
\[
x_i\coloneqq T^{m_i}\omega_i,
\qquad
x_{i,\epsilon}\coloneqq T^{m_i}\omega_{i,\epsilon}
\qquad(i=1,2).
\]
Since $T$ preserves $\cC$, for $i=1,2$ we have
\[
x_i,\qquad x_{i,\epsilon},\qquad
x_i-\epsilon x_{i,\epsilon}
=T^{m_i}(\omega_i-\epsilon\omega_{i,\epsilon})
\in\cC.
\]
Apply the four-point Ptolemy inequality, \cref{cor:mixed-ptolemy}, to
\[
x_1,\quad x_2,\quad x_{1,\epsilon},\quad x_{2,\epsilon}.
\]
Multiplying the resulting inequality by the sum of its two left-hand square roots yields
\begin{align*}
D_\epsilon
&\le
\sqrt{
 B(x_1,x_{1,\epsilon})B(x_2,x_{2,\epsilon})}\
\Bigl(
 \sqrt{
  B(x_1,x_2)B(x_{1,\epsilon},x_{2,\epsilon})}
 +
 \sqrt{
  B(x_1,x_{2,\epsilon})B(x_2,x_{1,\epsilon})}\ 
\Bigr),
\end{align*}
where
\begin{align*}
D_\epsilon
&\coloneqq
 \bigl|B(x_1,x_2)B(x_{1,\epsilon},x_{2,\epsilon})
 -B(x_1,x_{2,\epsilon})B(x_{1,\epsilon},x_2)\bigr|.
\end{align*}
Positivity of $Q$ on $\cC$ and $x_j-\epsilon x_{j,\epsilon}\in\cC$ give, for $i,j\in\{1,2\}$,
\begin{align*}
B(x_i,x_{j,\epsilon})
& \le 
B(x_i,\epsilon^{-1}x_j)
=\epsilon^{-1}B(x_i,x_j);\\
B(x_{i,\epsilon},x_{j,\epsilon})
& \le
B(\epsilon^{-1}x_i,\epsilon^{-1}x_j)
=\epsilon^{-2}B(x_i,x_j).
\end{align*}
Substituting these bounds into the preceding Ptolemy estimate yields
\begin{align}\label{eq:De<=BBB}
D_\epsilon
\le
2\epsilon^{-2}
B(x_1,x_2)
\sqrt{
 B(x_1,x_1)B(x_2,x_2)
}.
\end{align}
Finally, since $TN=NT$, one has
\begin{equation}\label{eq:perturbed-orbit-class}
x_{i,\epsilon}=x_i+\epsilon Nx_i.
\end{equation}
By bilinearity, expanding $D_\epsilon$ using \eqref{eq:perturbed-orbit-class} and canceling the cross terms gives that $\epsilon^{-2}D_\epsilon$ is exactly the left-hand side of \eqref{eq:uniform-infinitesimal-ptolemy}.
Dividing \eqref{eq:De<=BBB} by $\epsilon^2$ proves \eqref{eq:uniform-infinitesimal-ptolemy} with $C=2\epsilon^{-4}$.
\end{proof}

\subsection{Orbit-uniform Tur\'an upper estimate and Tur\'an degree detection}

We now translate the pairwise estimate \cref{lem:uniform-infinitesimal-ptolemy} into the multi-index notation for mixed $Q$-coefficients.

For $\gamma\in\Comp(d)$, 
write $\supp(\gamma)=\{j:\gamma_j>0\}$. 
Let $e_r$ denote the $r$-th standard basis vector of $\bR^d$.
For distinct $r,s$ with $\gamma_s>0$, write
\[
\gamma^{r\leftarrow s}
\coloneqq
\gamma+e_r-e_s
\]
for the elementary transfer of one unit from the $s$-th part to the $r$-th part.

For a polynomial $P=P(t_1,\ldots,t_d)\in\bR[\bt]$, define
\[
\cT_{rs}(P)\coloneqq (\partial_rP)(\partial_sP) - P\,\partial_r\partial_sP.
\]
For $r=s$, this is the usual differential Tur\'an expression $(\partial_rP)^2-P\,\partial_r^2P$.
On the real region $P\neq0$, $\cT_{rs}(P)/P^2$ is the $(r,s)$-entry of the Hessian of $-\log |P|$.

\begin{theorem}\label{thm:turan-upper}
Let $(V,\cC,Q,T)$ be a dynamical Hodge--Lorentz package of degree $d\ge2$, and write $T=Se^N=e^NS$ for the multiplicative Jordan--Chevalley decomposition of $T$.
Fix $h_1,\ldots,h_d\in\cC$, $\gamma\in\Comp(d)$, and distinct $r,s\in \supp(\gamma)$. 
Then there is a constant $C_{\gamma,r,s}>0$, depending only on the package, $\gamma$, $h_r$, and $h_s$, such that, for every $\bn\in\bZ_{\ge0}^d$,
\begin{equation}\label{eq:turan-upper}
\bigl|\cT_{rs}(\phi_{\gamma,\bn})(\bn)\bigr|
\le
C_{\gamma,r,s}\,\phi_{\gamma,\bn}(\bn)
\sqrt{
\phi_{\gamma^{r\leftarrow s},\bn}(\bn)
\phi_{\gamma^{s\leftarrow r},\bn}(\bn)
}.
\end{equation}
Here the $\phi_{*,\bn}$ are the unipotent-shift polynomials at $\bn$, as in \cref{def:unipotent-shift-poly}.
\end{theorem}

\begin{proof}
Fix $\bn=(n_1,\ldots,n_d)\in\bZ_{\ge0}^d$ and for simplicity denote
\[
x_j\coloneqq T^{n_j}h_j\in\cC,
\qquad 1\le j\le d.
\]
Using the following $d-2$ vectors as frozen arguments
\[
x_r^{[\gamma_r-1]},\qquad
x_s^{[\gamma_s-1]},\qquad
(x_j^{[\gamma_j]})_{j\ne r,s},
\]
where the last expression denotes the concatenated list with $x_j$ repeated $\gamma_j$ times for each $j\ne r,s$,
define the symmetric bilinear form on $V$ by
\[
B(u,v)
\coloneqq
Q\bigl(
u,v,
x_r^{[\gamma_r-1]},
x_s^{[\gamma_s-1]},
(x_j^{[\gamma_j]})_{j\ne r,s}
\bigr).
\]

Apply \cref{lem:uniform-infinitesimal-ptolemy} with
\[
\omega_1=h_r,\qquad
\omega_2=h_s,\qquad
m_1=n_r,\qquad
m_2=n_s.
\]
It gives a constant $C_{r,s}>0$, depending only on the package and on $h_r,h_s$, such that
\begin{align}\label{eq:pairwise-mixed-turan}
\begin{split}
& \bigl|B(x_r,x_s)B(Nx_r,Nx_s) - B(Nx_r,x_s)B(x_r,Nx_s)\bigr|\\
& \qquad\le C_{r,s} B(x_r,x_s)\sqrt{B(x_r,x_r)B(x_s,x_s)}.
\end{split}
\end{align}
Crucially, $C_{r,s}$ is independent of $\bn$. 

Note that
\[
\phi_{\gamma,\bn}(\bt)=Q_\gamma\bigl(e^{(t_1-n_1)N}x_1,\ldots,e^{(t_d-n_d)N}x_d\bigr) \in \bR[t_1,\dots,t_d].
\]
By multilinearity of $Q$ 
and taking partial derivatives, we have
\begin{align*}
 \phi_{\gamma,\bn}(\bn)&=B(x_r,x_s), \qquad\qquad \partial_r\phi_{\gamma,\bn}(\bn) =\gamma_rB(Nx_r,x_s),\\
\partial_s\phi_{\gamma,\bn}(\bn)&=\gamma_sB(x_r,Nx_s), \quad \ \partial_r\partial_s\phi_{\gamma,\bn}(\bn) =\gamma_r\gamma_sB(Nx_r,Nx_s).
\end{align*}
Moreover, the definition of the elementary transfers gives
\[
\begin{aligned}
B(x_r,x_r)
=\phi_{\gamma^{r\leftarrow s},\bn}(\bn),\qquad
B(x_s,x_s)
=\phi_{\gamma^{s\leftarrow r},\bn}(\bn).
\end{aligned}
\]
Substituting these identities into \eqref{eq:pairwise-mixed-turan} and multiplying by $\gamma_r\gamma_s$ therefore yield
\[
\bigl|\cT_{rs}(\phi_{\gamma,\bn})(\bn)\bigr|
\le
C_{\gamma,r,s}\,\phi_{\gamma,\bn}(\bn)
\sqrt{
\phi_{\gamma^{r\leftarrow s},\bn}(\bn)
\phi_{\gamma^{s\leftarrow r},\bn}(\bn)
},
\]
where $C_{\gamma,r,s}\coloneqq\gamma_r\gamma_sC_{r,s}$ depends only on the package, $\gamma$, $h_r$, $h_s$, and is independent of $\bn$.
\end{proof}

For a polynomial $P\in\bR[t_1,\dots,t_d]$ of total degree $M>0$, every pairwise Tur\'an polynomial $\cT_{rs}(P)$
has total degree at most $2M-2$.
The next lemma says that diagonal translation invariance prevents simultaneous cancellation of the degree-$(2M-2)$ terms.
Consequently, at least one off-diagonal Tur\'an polynomial has degree exactly $2M-2$.

\begin{lemma}
\label{lem:turan-degree-detection}
Let $d\ge2$ and let $P\in\bR[t_1,\ldots,t_d]$ be a polynomial of positive total degree.
Assume that
\[
P(\bt+c\bone)=P(\bt)
\]
for every $\bt\in\bR^d$ and $c\in\bR$, where $\bone=(1,\ldots,1)$.
Then there exist distinct $r,s$ such that
\[
\deg_{\tot}\cT_{rs}(P)=2\deg_{\tot} P-2.
\]
\end{lemma}

\begin{proof}
Set $M\coloneqq \deg_{\tot}P\in \bZ_{>0}$.
Let $P_M$ be the top homogeneous part of $P$.
Recall that
\[
\cT_{ij}(P_M) \coloneqq (\partial_iP_M)(\partial_jP_M) - P_M\partial_i\partial_jP_M.
\]
Diagonal translation invariance gives $\sum_j\partial_jP=0$ and hence $\sum_j\partial_jP_M=0$.
It follows that for every $i$,
\begin{align}
    \sum_j\cT_{ij}(P_M)=0. \label{eq: sum Tij=0}
\end{align}
Euler's identities give
\begin{align}
\sum_{1\le i\le d}t_i\partial_iP_M=MP_M,
\qquad
\sum_{1\le i,\, j\le d}t_it_j\partial_i\partial_jP_M=M(M-1)P_M. \label{eq: eulerx2}
\end{align}
Consequently, we have
\begin{align*}
\frac12\sum_{1\le i,\, j\le d}(t_i-t_j)^2\cT_{ij}(P_M)
&=-\sum_{1\le i,\, j\le d}t_it_j\cT_{ij}(P_M)\\
&=P_M\sum_{1\le i,\, j\le d}t_it_j\partial_i\partial_jP_M - \Biggl(\sum_{1\le i\le d}t_i\partial_iP_M\Biggr)^2\\
&=-MP_M^2\neq 0,
\end{align*}
where the first equality is by \eqref{eq: sum Tij=0} and the last is by \eqref{eq: eulerx2}.
Hence there exist $r\ne s$ such that $\cT_{rs}(P_M)\ne0$.
Note that the degree-$(2M-2)$ homogeneous component of $\cT_{rs}(P)$ is exactly $\cT_{rs}(P_M)$.
Hence $\deg_{\tot}\cT_{rs}(P)=2M-2$.
\end{proof}

\section{From Tur\'an estimates to growth bounds}\label{sec:abstract-growth}


\subsection{Unipotent two-block case: discrete Laplacian and degree-growth upper bounds}\label{subsec:two-block-case}

The mechanism underlying the discrete Laplacian upper bound $4$ and the quadratic upper bound for degree growth is most transparent in the unipotent two-block case.
As a warm-up, we first show how to use the orbit-uniform Tur\'an upper estimate to control the discrete Laplacian and to obtain the global upper bound via a discrete maximum principle.


\begin{theorem}
\label{thm:two-block-model}
Let $(V,\cC,Q,T=e^N)$ be a unipotent dynamical Hodge--Lorentz package of degree $d\ge2$ and fix $h_1,h_2\in\cC$.
For $0\le i\le d$, set
\[
\phi_i(t)
\coloneqq
Q\bigl((e^{tN}h_1)^{[i]},h_2^{[d-i]}\bigr) \in\bR[t],
\qquad
\mu_i\coloneqq\deg \phi_i.
\]
Then $\mu_i=\mu_{d-i}\in2\bZ_{\ge0}$ for every $0\le i\le d$, and, for every
$1\le i\le d-1$,
\begin{equation}\label{eq:two-block-curvature}
0
\le
\Delta\mu_i=2\mu_i-\mu_{i-1}-\mu_{i+1}
\le
4.
\end{equation}
In particular,
\[
\Delta\mu_i\in\{0,2,4\}.
\]
Consequently, for every $0\le i\le d$,
\[
\mu_i\le2i(d-i).
\]
\end{theorem}

We first prove the parity assertion for general mixed orbit functions.

\begin{lemma}
\label{lem:unipotent-mixed-parity}
Let $(V,\cC,Q,T=e^N)$ be a unipotent dynamical Hodge--Lorentz package of degree $d\ge2$. Fix $\bh=(h_1,\ldots,h_d)\in\cC^d$ and $\gamma\in\Comp(d)$.
Then the canonical polynomial extension of the mixed orbit function $\Phi_{T,\gamma,\bh}$ defined in \cref{def:mixed-orbit-functions} has even total degree:
\[
\mu_\gamma\coloneqq\deg_{\tot}\Phi_{T,\gamma,\bh}\in2\bZ_{\ge0}.
\]
\end{lemma}

\begin{proof}
Denote  $P(\bt)\coloneqq \Phi_{T,\gamma,\bh}(\bt)\in \bR[\bt]$ for simplicity.
Since $N\in \stab(Q)$, diagonal $N$-invariance \cref{lem:stabQ} gives
\[
P(\bt+c\bone)=P(\bt)
\qquad(\bt\in\bR^d,\ c\in\bR).
\]
For any $\bn\in\bZ^d$, choose an integer $k\ge1$ sufficiently large such that $\bn+k\bone\in\bZ_{\ge0}^d$.
Hence $T^{\bn+k\bone}\bh\in \cC^d$.
Cone positivity then gives
\[
P(\bn)=P(\bn+k\bone)=Q_\gamma(T^{\bn+k\bone}\bh)>0.
\]
Thus $P$ is positive on $\bZ^d$.
Let $P_{\mu_\gamma}$ be the top homogeneous part of $P$.
Choose $\ba\in\bZ^d$ with $P_{\mu_\gamma}(\ba)\ne0$.
If $\mu_\gamma$ were odd, then
\[
P(n\ba)=n^{\mu_\gamma}P_{\mu_\gamma}(\ba)+O(|n|^{\mu_\gamma-1})
\]
would have opposite signs as $n\to\pm \infty$, 
contradicting
$P(n\ba)>0$ for every $n\in\bZ$.  Hence $\mu_\gamma$ is even.
\end{proof}

We then recall an elementary comparison that will also be used on maximal affine intervals in \cref{thm:abstract-turan-dichotomy}; see its general graph form in \cref{lem:transfer-maximum-principle}.

\begin{lemma}
\label{lem:discrete-comparison}
Let $r<s$ be integers, and let $(a_i)_{i=r}^s$ and $(b_i)_{i=r}^s$ be real sequences.
If
\[
a_r\le b_r,
\qquad
a_s\le b_s,
\qquad
\Delta a_i\le\Delta b_i
\quad(r<i<s),
\]
then $a_i\le b_i$ for every $r\le i\le s$.
\end{lemma}

\begin{proof}
Set $\delta_i=b_i-a_i$.
Then $\delta_r,\delta_s\ge0$ and $\Delta\delta_i\ge0$ for $r<i<s$.\footnote{Alternatively, one can readily deduce the assertion from the fact that a concave sequence lies above the line segment joining its endpoints. Here we use proof by contradiction as an exact parallel of the proof of \cref{lem:transfer-maximum-principle}.} Suppose that $\delta_i<0$ for some $i$, we may choose $j$ maximal with the property
\[
\delta_j=\min\{\delta_i\mid r\leq i\leq s\}<0.
\]
Then $r<j<s$ and 
\[
2\delta_j-\delta_{j-1}-\delta_{j+1}=\Delta\delta_j\ge0.
\]
By the minimality of $\delta_j$, we have $\delta_j=\delta_{j-1}=\delta_{j+1}$, which contradicts the maximality of $j$. 
Therefore $\delta_i\geq 0$ for all $r\leq i\leq s$ and hence $a_i\leq b_i$.
\end{proof}

\begin{proof}[Proof of \cref{thm:two-block-model}]
Since $N$ is nilpotent, every $\phi_i$ is a polynomial.
The parity assertion is the two-block specialization of
\cref{lem:unipotent-mixed-parity}.
By \cref{lem:stabQ} and openness of the cone $\cC$,
\[
\phi_i(n)=Q((e^{nN}h_1)^{[i]},h_2^{[d-i]})=Q(h_1^{[i]},(e^{-nN}h_2)^{[d-i]}) \asymp \phi_{d-i}(-n).
\]
Hence $\mu_i=\mu_{d-i}$.
Note also that both $\phi_0(t)=Q(h_2^{[d]})$ and $\phi_d(t)=Q((e^{tN}h_1)^{[d]})=Q(h_1^{[d]})$ are constants.
So $\mu_0=\mu_d=0$.

For every $0\le i\le d$ and every $n\in\bZ_{\ge0}$, the vectors $T^nh_1$ and $h_2$ belong to $\cC$.
Hence $\phi_i(n)>0$, and every $\phi_i$ has positive leading coefficient.
By the log-concavity of mixed $Q$-coefficients, \cref{cor:mixed-log-concavity},
\[
\phi_i(n)^2\ge \phi_{i-1}(n)\phi_{i+1}(n).
\]
Comparing degrees yields $\Delta\mu_i\ge0$.
If $\mu_i=0$ for some $1\le i\le d-1$, then this inequality and $\mu_j\ge0$ force all $\mu_j$ to vanish.
We may therefore assume that $\mu_i>0$ for every $1\le i\le d-1$.

We apply \cref{thm:turan-upper} to our unipotent dynamical Hodge--Lorentz package $(V,\cC,Q,T=e^N)$, $\gamma=(i,d-i,0,\dots,0)$, $\bh=(h_1,h_2,h_2,\ldots,h_2)$, and $\bn=(n,0,\dots,0)$.
Then there is a constant $C>0$, independent of $n$, such that
\begin{align}\label{eq:two-block-bilinear-turan}
\bigl|\cT_{12}(\phi_{\gamma,\bn})(\bn)\bigr|
\le
C\,\phi_{\gamma,\bn}(\bn)
\sqrt{
\phi_{\gamma^{1\leftarrow 2},\bn}(\bn)
\phi_{\gamma^{2\leftarrow 1},\bn}(\bn)}.
\end{align}
Since we are in the two-block case $\gamma=(i,d-i,0,\dots,0)$,
\begin{align*}
\phi_{\gamma,\bn}(t_1,t_2, 0,\dots, 0)&=\phi_i(t_1-t_2),\\
\phi_{\gamma^{1\leftarrow 2},\bn}(t_1,t_2, 0,\dots, 0)&=\phi_{i+1}(t_1-t_2),\\
\phi_{\gamma^{2\leftarrow 1},\bn}(t_1,t_2, 0,\dots, 0)&=\phi_{i-1}(t_1-t_2).
\end{align*}
Differentiating at $(t_1,t_2)=(n,0)$ yields
\begin{align*}
\partial_1\phi_{\gamma,\bn}(n, 0,\dots, 0)&=\phi'_i(n),\\
\partial_2\phi_{\gamma,\bn}(n, 0,\dots, 0)&=-\phi'_i(n),\\
\partial_1\partial_2\phi_{\gamma,\bn}(n, 0,\dots, 0)&=-\phi''_i(n).
\end{align*}
Substituting all these identities into \eqref{eq:two-block-bilinear-turan}, we obtain
\begin{equation}\label{eq:two-block-one-variable-turan}
\left|(\phi_i'(n))^2-\phi_i(n)\phi_i''(n)\right|
\le C\,\phi_i(n)\sqrt{\phi_{i-1}(n)\phi_{i+1}(n)}.
\end{equation}

For each $i$, write $\phi_i(t)=a_it^{\mu_i}+O(t^{\mu_i-1})$ with $a_i>0$.
Since $\mu_i>0$,
\[
(\phi_i'(t))^2-\phi_i(t)\phi_i''(t)
\sim
\mu_i a_i^2t^{2\mu_i-2}.
\]
Comparing powers of $n$ in \eqref{eq:two-block-one-variable-turan} gives
\[
2\mu_i-2
\le
\mu_i+\frac12(\mu_{i-1}+\mu_{i+1}),
\]
and hence $\Delta\mu_i\le4$.
Together with the lower bound, this proves \eqref{eq:two-block-curvature}.

Finally, set $b_i\coloneqq2i(d-i)$.
Since $\mu_0=b_0=\mu_d=b_d=0$, and $\Delta\mu_i\le4=\Delta b_i$, the discrete maximum principle, \cref{lem:discrete-comparison}, gives $\mu_i\le b_i=2i(d-i)$.
\end{proof}

\subsection{Unipotent mixed-orbit growth}\label{subsec:unipotent-mixed}


The main result of this subsection is the abstract counterpart of \cref{thm:zero-entropy-mixed}, which gives a sharp upper bound for the mixed orbit function in the unipotent case.

\begin{theorem}
\label{thm:unipotent-mixed-orbit-bound}
Let $(V,\cC,Q,T=e^N)$ be a unipotent dynamical Hodge--Lorentz package of degree $d\ge2$.
Then for any $\bh=(h_1,\ldots,h_d)\in\cC^d$ and any $\gamma\in\Comp(d)$, the canonical polynomial extension of the mixed orbit function $\Phi_{T,\gamma,\bh}$ to $\bR^d$ satisfies
\[
\deg_{\tot}\Phi_{T,\gamma,\bh}\in2\bZ_{\ge0},
\qquad
\deg_{\tot}\Phi_{T,\gamma,\bh}
\le
b_\gamma
\coloneqq
d^2-\|\gamma\|_2^2.
\]
Consequently, for every $\bl=(\ell_1,\ldots,\ell_d)\in\bZ^d$,
\[
\Phi_{T,\gamma,\bh}(n\bl)
=
Q_\gamma\bigl(
T^{\ell_1n}h_1,\ldots,T^{\ell_dn}h_d
\bigr)
=
O(n^{b_\gamma})
\qquad(n\to\infty).
\]
\end{theorem}

Throughout the proof, fix a unipotent dynamical Hodge--Lorentz package $(V,\cC,Q,T=e^N)$ of degree $d\ge2$.
The transfer graph has vertex set $\Comp(d)$, and two vertices are adjacent if they differ by an elementary transfer.
For the local argument, if $r\ne s$ belong to $\supp(\gamma)$, the two opposite transfers associated with the unordered pair $\{r,s\}$ are
\[
\gamma^{r\leftarrow s}=\gamma+e_r-e_s,
\qquad
\gamma^{s\leftarrow r}=\gamma+e_s-e_r.
\]

Combining the Tur\'an degree-detection \cref{lem:turan-degree-detection} with the orbit-uniform Tur\'an upper estimate \cref{thm:turan-upper} yields the following local transfer inequality.

\begin{lemma}
\label{lem:local-transfer}
Fix $\bh=(h_1,\ldots,h_d)\in\cC^d$.
For every $\eta\in\Comp(d)$, set
\[
\Phi_\eta(\bt)
\coloneqq
Q_\eta(e^{t_1N}h_1,\ldots,e^{t_dN}h_d),
\qquad
\mu_\eta
\coloneqq
\deg_{\tot}\Phi_\eta.
\]
For every $\gamma\in\Comp(d)$ with $\mu_\gamma>0$, there exist distinct $r,s\in\supp(\gamma)$ such that
\[
2\mu_\gamma
-
\mu_{\gamma^{r\leftarrow s}}
-
\mu_{\gamma^{s\leftarrow r}}
\le4.
\]
\end{lemma}

\begin{proof}
Since $N\in\stab(Q)$, one has
\[
\Phi_\gamma(\bt+c\bone)=\Phi_\gamma(\bt)
\]
for every $\bt\in\bR^d$ and $c\in\bR$.
Moreover, if $j\notin\supp(\gamma)$, then $\Phi_\gamma$ is independent of $t_j$, and hence
\[
\cT_{jk}(\Phi_\gamma)=0
\qquad\text{for every }k.
\]
Therefore \cref{lem:turan-degree-detection} gives distinct $r,s\in\supp(\gamma)$ such that $\cT_{rs}(\Phi_\gamma)\neq0$ and
\[
\deg_{\tot}\cT_{rs}(\Phi_\gamma)
=
2\mu_\gamma-2.
\]

Since $T=e^N$ is unipotent, the unipotent-shift polynomials $\phi_{\eta,\bn}(\bt)=\Phi_\eta(\bt)$ for every $\eta\in\Comp(d)$ and $\bn\in\bZ_{\ge0}^d$; see \cref{def:unipotent-shift-poly}.
Thus \cref{thm:turan-upper} gives, uniformly in $\bn$,
\[
\bigl|\cT_{rs}(\Phi_\gamma)(\bn)\bigr|
\le
C\Phi_\gamma(\bn)
\sqrt{
\Phi_{\gamma^{r\leftarrow s}}(\bn)
\Phi_{\gamma^{s\leftarrow r}}(\bn)
}.
\]

Let $R\neq0$ be the top homogeneous part of $\cT_{rs}(\Phi_\gamma)$.
Choose $\ba\in\bZ_{>0}^d$ with $R(\ba)\ne0$, possible because the positive integer lattice is Zariski dense in $\bR^d$.
Indeed, a polynomial vanishing on $\bZ_{>0}^d$ is zero, by induction on $d$ from the one-variable case.
Then
\[
\bigl|\cT_{rs}(\Phi_\gamma)(n\ba)\bigr|
\asymp
n^{2\mu_\gamma-2}.
\]
On the other hand, 
$\Phi_\eta(n\ba)=O(n^{\mu_\eta})$
for every $\eta\in\Comp(d)$.
Substituting $\bn=n\ba$ into the preceding estimate and comparing powers of $n$ yields
\[
2\mu_\gamma-2
\le
\mu_\gamma+
\frac12\left(
\mu_{\gamma^{r\leftarrow s}}
+
\mu_{\gamma^{s\leftarrow r}}
\right),
\]
which is equivalent to the asserted inequality.
\end{proof}

Finally, a transfer-graph maximum principle will give us the desired upper bound.

\begin{lemma}
\label{lem:transfer-maximum-principle}
Let 
\begin{align*}
\nu, \,\nu'\colon\Comp(d)\to\bR, \quad \gamma\mapsto\nu_\gamma, \,\nu'_\gamma,
\end{align*}
be two functions.
Suppose that for every $\gamma \in \Comp(d)$ with $\nu_\gamma>\nu'_\gamma$, there exist distinct $r,s\in\supp(\gamma)$ such that
\[
2\nu_\gamma
-
\nu_{\gamma^{r\leftarrow s}}
-
\nu_{\gamma^{s\leftarrow r}}
\leq 2\nu'_\gamma
-
\nu'_{\gamma^{r\leftarrow s}}
-
\nu'_{\gamma^{s\leftarrow r}}.
\]
Then for every $\gamma\in\Comp(d)$,
\[
\nu_\gamma\le \nu'_\gamma.
\]
\end{lemma}
\begin{proof}
Set
\[
\delta_\gamma
\coloneqq
\nu'_\gamma-\nu_\gamma.
\]
Suppose that $\min_\gamma\delta_\gamma<0$.
Since $\Comp(d)\subset \bZ_{\geq 0}^d$ is finite, we choose $\gamma$ lexicographically maximal among the global minimizers of $\delta$.
Since $\delta_\gamma<0$, we have $\nu_\gamma>\nu'_\gamma$, so the assumption gives distinct $r,s\in\supp(\gamma)$ such that
\[
2\delta_\gamma
\ge
\delta_{\gamma^{r\leftarrow s}}
+
\delta_{\gamma^{s\leftarrow r}}.
\]
By the minimality of $\delta_\gamma$, 
we have $\delta_\gamma
=
\delta_{\gamma^{r\leftarrow s}}
=
\delta_{\gamma^{s\leftarrow r}}$.
But one of $\gamma^{r\leftarrow s}$ and ${\gamma^{s\leftarrow r}}$ is lexicographically larger than $\gamma$, contradicting the lexicographic maximality of $\gamma$.
Therefore $\min_\gamma\delta_\gamma\ge0$, and hence $\nu_\gamma\le \nu'_\gamma$ for every $\gamma\in\Comp(d)$.
\end{proof}

\begin{proof}[Proof of \cref{thm:unipotent-mixed-orbit-bound}]
Fix $\bh\in\cC^d$.
For any $\gamma\in\Comp(d)$, abbreviate $\Phi_{T,\gamma,\bh}$ to $\Phi_\gamma$.
Since $N$ is nilpotent, $\Phi_{\gamma}$ is a polynomial.
Set
\[
\mu_\gamma=\deg_{\tot}\Phi_\gamma.
\]
By \cref{lem:unipotent-mixed-parity}, $\mu_\gamma$ is even.
If $\mu_\gamma>b_\gamma\geq0$, then \cref{lem:local-transfer} asserts that there exist distinct $r,s\in\supp(\gamma)$ such that
\[
2\mu_\gamma
-
\mu_{\gamma^{r\leftarrow s}}
-
\mu_{\gamma^{s\leftarrow r}}
\le4.
\]
On the other hand, note that
\[
2b_\gamma-b_{\gamma^{r\leftarrow s}}-b_{\gamma^{s\leftarrow r}} = 2b_\gamma-(b_\gamma-2(\gamma_r-\gamma_s+1))-(b_\gamma-2(\gamma_s-\gamma_r+1)) = 4.
\]
Hence by \cref{lem:transfer-maximum-principle}, $\mu_\gamma\le b_\gamma$ for all $\gamma$.
Substituting $\bt=n\bl$ gives the ray bound.
\end{proof}

\subsection{General two-block case: peripheral polynomial growth}\label{subsec:peripheral-two-block}

The main result of this subsection is the abstract counterpart of \cref{thm:endomorphism-peripheral-growth}.

\begin{theorem}
\label{thm:abstract-turan-dichotomy}
Let $(V,\cC,Q,T)$ be a dynamical Hodge--Lorentz package of degree $d\ge2$.
Fix $h_1,h_2\in\cC$.
Assume that, for every $0\le i\le d$, there are $\lambda_i>0$ and $\mu_i\in\bZ_{\ge0}$ such that
\[
D_i(n)\coloneqq
Q\bigl((T^nh_1)^{[i]},h_2^{[d-i]}\bigr)\asymp\lambda_i^n n^{\mu_i} \qquad (n\to +\infty).
\]
Then for every $1\le i\le d-1$, one has $\lambda_i^2\ge\lambda_{i-1}\lambda_{i+1}$.
Moreover, the following dichotomy holds:
\[
\begin{cases}
\mu_i=0,
&\text{if }\lambda_i^2>\lambda_{i-1}\lambda_{i+1},\\
0\le \Delta\mu_i = 2\mu_i-\mu_{i-1}-\mu_{i+1}\le4,
&\text{if }\lambda_i^2=\lambda_{i-1}\lambda_{i+1}.
\end{cases}
\]
Consequently, on every maximal affine interval $[r,s]$ of $(\log\lambda_i)_i$ (see \cref{def:maximal-affine-interval}),
\[
\mu_i\le2(i-r)(s-i)
\qquad(r\le i\le s).
\]
In particular, $\mu_i\le2i(d-i)$ for every $0\le i\le d$.
\end{theorem}

As in \cref{lem:multi-JC}, write the multiplicative Jordan--Chevalley decomposition of $T$ by
\[
T=Se^N=e^NS.
\]
For $0\le i\le d$ and $n\in\bZ_{\geq0}$, define the one-variable unipotent-shift polynomial at $n$ by
\[
\phi_{i,n}(t)
\coloneqq
Q\bigl((e^{tN}S^nh_1)^{[i]},h_2^{[d-i]}\bigr)
=
Q\bigl((e^{(t-n)N}T^nh_1)^{[i]},h_2^{[d-i]}\bigr) \in \bR[t].
\]
Note that $\phi_{i,n}(n)=D_i(n)>0$ for every $n\in\bZ_{\geq0}$ by cone positivity.

\begin{lemma}
\label{lem:turan-lower}
For each $0\le i\le d$, if
\[
\phi_{i,n}(n)=D_i(n) \asymp \lambda_i^n n^{\mu_i}
\]
for some $\lambda_i>0$ and $\mu_i\in\bZ_{\ge0}$, then
\begin{equation}\label{eq:shifted-log-turan}
\phi_{i,n}'(n)^2-\phi_{i,n}(n)\phi_{i,n}''(n) = 
{n^{-2}\phi_{i,n}(n)^2}
\bigl(\mu_i + O(n^{-1})\bigr).
\end{equation}
In particular, if $\mu_i>0$, then there exists $C_i>0$ such that
\begin{equation}\label{eq:uniform-Turan-lower-bound}
\phi_{i,n}'(n)^2-\phi_{i,n}(n)\phi_{i,n}''(n)
\ge
C_i\lambda_i^{2n}n^{2\mu_i-2}
\end{equation}
for all sufficiently large $n$.
\end{lemma}

\begin{proof}
Fix $0\le i\le d$.
Note that $S$ is diagonalizable after complexification.
We can decompose $h_1$ as a sum of eigenvectors of $S$.
Since $N$ commutes with $S$, multilinearity of $Q$ gives the following expression
\begin{equation}\label{eq:phiint-expansion}
\phi_{i,n}(t) = \sum_{\lambda\in\Lambda_i}\lambda^nP_\lambda(t),
\end{equation}
where $\Lambda_i\subset\bC^\times$ is a finite set of products of eigenvalues of $S$ (or $T$) and the nonzero polynomials $P_\lambda\in\bC[t]$ are independent of $n$.
Denote the leading coefficient of each $P_\lambda$ by $c_{\lambda}\in\bC^*$.
Set
\[
\rho\coloneqq \max_{\lambda\in\Lambda_i}|\lambda|,
\qquad
k\coloneqq \max_{|\lambda|=\rho}\deg P_\lambda,
\]
and define
\[
a_n
\coloneqq
\sum_{\substack{|\lambda|=\rho\\ \deg P_\lambda=k}}
c_\lambda
\left(\frac{\lambda}{\rho}\right)^n.
\]
Since $\overline{c_\lambda}=c_{\overline{\lambda}}$ and the set $\{\lambda\in\Lambda_i:|\lambda|=\rho, \, \deg P_\lambda=k\}$ is stable under complex conjugation, $a_n$ is real.
Ces\'aro orthogonality then gives
\[
\lim_{M\to+\infty}
\frac{1}{M}\sum_{n=0}^{M-1}|a_n|^2
=\sum_{\substack{|\lambda|=\rho\\[2pt] \deg P_\lambda=k}}
|c_\lambda|^2
>0.
\]
In particular, $a_n$ does not converge to zero, i.e., $a_n\ne o(1)$.

Grouping the terms in the spectral--polynomial expansion \eqref{eq:phiint-expansion} of $\phi_{i,n}(t)$ according to their modulus and degree, and using the definitions of $\rho$, $k$, and $a_n$, we obtain
\begin{equation*}\label{eq:phiinn-asymp}
\phi_{i,n}(n) = \rho^n n^k\bigl(a_n+O(n^{-1})\bigr).
\end{equation*}
Combining this with $\phi_{i,n}(n)\asymp\lambda_i^n n^{\mu_i}$ gives
\begin{equation*}\label{eq:an-asymp}
a_n+O(n^{-1}) = \frac{\phi_{i,n}(n)}{\rho^n n^k}
\asymp \left(\frac{\lambda_i}{\rho}\right)^n n^{\mu_i-k}.
\end{equation*}
It thus follows from the boundedness of $a_n$ (by definition) and $a_n\neq o(1)$ that
\begin{equation}\label{eq:exact-lambda-mu}
\rho=\lambda_i \quad\text{and}\quad k=\mu_i.
\end{equation}

Differentiating the polynomial $\phi_{i,n}(t)$ at $n$ and using \eqref{eq:exact-lambda-mu} give
\begin{align*}
\phi'_{i,n}(n)&=\lambda_i^n n^{\mu_i-1}\bigl(\mu_ia_n+O(n^{-1})\bigr),\\
\phi''_{i,n}(n)&=\lambda_i^n n^{\mu_i-2}\bigl(\mu_i(\mu_i-1)a_n+O(n^{-1})\bigr).
\end{align*}
By a direct computation, we get
\[
\phi_{i,n}'(n)^2-\phi_{i,n}(n)\phi_{i,n}''(n) = \lambda_i^{2n}n^{2\mu_i-2}\bigl(\mu_ia_n^2+O(n^{-1})\bigr),
\]
while the boundedness of $a_n$ yields
\[
\mu_i n^{-2}\phi_{i,n}(n)^2 = \mu_i \lambda_i^{2n} n^{2\mu_i-2} \bigl(a_n+O(n^{-1})\bigr)^2 = \lambda_i^{2n} n^{2\mu_i-2} \bigl(\mu_i a_n^2+O(n^{-1})\bigr).
\]
Taking the difference of the above two displayed equations gives
\[
\phi_{i,n}'(n)^2-\phi_{i,n}(n)\phi_{i,n}''(n) - \mu_i n^{-2}\phi_{i,n}(n)^2 = O\bigl(\lambda_i^{2n} n^{2\mu_i-3}\bigr).
\]
Since $\phi_{i,n}(n)\asymp\lambda_i^n n^{\mu_i}$, one has $\lambda_i^{2n} n^{2\mu_i-3} \asymp n^{-3}\phi_{i,n}(n)^2$.
The asserted \eqref{eq:shifted-log-turan} thus follows.

If $\mu_i>0$, then as $\phi_{i,n}(n)\asymp \lambda_i^{n}n^{\mu_i}$, the right-hand side of \eqref{eq:shifted-log-turan} satisfies
\[
 {n^{-2}\phi_{i,n}(n)^2} \bigl(\mu_i + O(n^{-1})\bigr)\asymp \lambda_i^{2n}n^{2\mu_i-2}.
\]
Hence there is a constant $C_i>0$ such that \eqref{eq:uniform-Turan-lower-bound} holds for all sufficiently large $n$.
\end{proof}

\begin{proof}[Proof of \cref{thm:abstract-turan-dichotomy}]
Fix $1\le i\le d-1$.
The log-concavity of mixed $Q$-coefficients, \cref{cor:mixed-log-concavity}, gives
\[
D_i(n)^2\ge D_{i-1}(n)D_{i+1}(n).
\]
Since $D_i(n)\asymp\lambda_i^n n^{\mu_i}$ by assumption, comparison first of exponential rates and then, in the equality case, of polynomial exponents gives
\[
\lambda_i^2\ge\lambda_{i-1}\lambda_{i+1},
\qquad
2\mu_i\ge\mu_{i-1}+\mu_{i+1}
\quad\text{if equality holds.}
\]
We then apply \cref{thm:turan-upper} with
\[
\gamma=(i,d-i,0,\ldots,0),
\qquad
\bh=(h_1,h_2,h_2,\ldots,h_2),
\qquad
\bn=(n,0,\ldots,0).
\]
Diagonal $N$-invariance \cref{lem:stabQ} gives
\begin{align*}
\phi_{\gamma,\bn}(t_1,t_2,0,\ldots,0)&=Q\bigl((e^{t_1N}S^nh_1)^{[i]},(e^{t_2N}h_2)^{[d-i]}\bigr)\\
&=Q\bigl((e^{(t_1-t_2)N}S^nh_1)^{[i]},h_2^{[d-i]}\bigr)=
\phi_{i,n}(t_1-t_2),\\
\phi_{\gamma^{1\leftarrow 2},\bn}(t_1,t_2,0,\ldots,0)&= \phi_{i+1,n}(t_1-t_2),\\
\phi_{\gamma^{2\leftarrow 1},\bn}(t_1,t_2,0,\ldots,0)&= \phi_{i-1,n}(t_1-t_2).
\end{align*}
Now by evaluating the corresponding \eqref{eq:turan-upper} at $(t_1,t_2,0,\dots,0)=(n,0,\dots,0)$, we get
\begin{equation}\label{eq:shifted-turan-upper}
\left|
\phi_{i,n}'(n)^2-\phi_{i,n}(n)\phi_{i,n}''(n)
\right|
\lesssim
D_i(n)\sqrt{D_{i-1}(n)D_{i+1}(n)}.
\end{equation}

First consider the case when $\lambda_i^2>\lambda_{i-1}\lambda_{i+1}$.
Suppose the contrary that $\mu_i>0$.
By assumption, $D_i(n)\asymp\lambda_i^n n^{\mu_i}$, combining \eqref{eq:shifted-turan-upper} with \eqref{eq:uniform-Turan-lower-bound} in \cref{lem:turan-lower} yields
\begin{equation}\label{eq:lambda-mu-comparison}
\lambda_i^{2n}n^{2\mu_i-2}
\lesssim
\lambda_i^n(\lambda_{i-1}\lambda_{i+1})^{n/2}
n^{\mu_i+(\mu_{i-1}+\mu_{i+1})/2}.
\end{equation}
Strict log-concavity makes the left-hand side exponentially larger, a contradiction, so $\mu_i=0$.

Then consider the case when $\lambda_i^2=\lambda_{i-1}\lambda_{i+1}$.
If $\mu_i=0$, then concavity of $\mu_i$ forces $\mu_{i-1}=\mu_{i+1}=0$ and hence $\Delta \mu_i=0$.
We may therefore assume that $\mu_i>0$.
Now \eqref{eq:lambda-mu-comparison} still holds by \eqref{eq:uniform-Turan-lower-bound}.
Cancelling out the exponential parts and comparing powers of $n$ give $\Delta \mu_i\le4$.
Together with concavity, this proves the dichotomy.

Finally, let $[r,s]\subset [0,d]$ be a maximal affine interval of $(\log\lambda_i)_i$.
Because $D_0(n)=Q(h_2^{[d]})$ is constant and $D_d(n)=\chi(T)^nQ(h_1^{[d]})$ is purely exponential, $\mu_0=\mu_d=0$.
By the maximality of $[r,s]$, we have either $r=0$, or $r>0$ and $\lambda_r^2>\lambda_{r-1}\lambda_{r+1}$, so $\mu_r=0$ in either case. Similarly, $\mu_s=0$.
Applying the discrete maximum principle \cref{lem:discrete-comparison}
to $(\mu_i)_i$ and $(2(i-r)(s-i))_i$ gives $\mu_i\le 2(i-r)(s-i)$.
The inequality $2(i-r)(s-i)\le 2i(d-i)$ yields the global bound.
\end{proof}

\section{Applications to algebraic dynamics}\label{sec:applications}


\subsection{Degree--norm comparison}

Degree--norm comparison is standard in characteristic zero; see, e.g., \cite[Lemme~4]{DS05a}.
We give a cone-theoretic proof for endomorphisms in arbitrary characteristic, using the positivity properties of the pliant cone established in \cite{FL17a}.
See also \cite[Theorem~2]{Dang20} for dominant rational self-maps.

\begin{proposition}
\label{prop:degree-norm-comparison}
Let $X$ be a normal projective variety of dimension $d$ over $\bk$, let $f\colon X\to X$ be a surjective endomorphism, and let $H$ be an ample divisor on $X$.
Then for every $0\le i\le d$, the limit
\[
\lambda_i(f) \coloneqq \lim_{n\to +\infty} \bigl(\deg_i(f^n)\bigr)^{1/n}
\]
exists and equals the spectral radius of $f^*|_{\N^i(X)_\bR}$.

Let $\mu_i+1$ be the largest size of Jordan blocks of $f^*|_{\N^i(X)_\bC}$ associated with eigenvalues of modulus $\lambda_i(f)$.
Then for any norm on $\N^i(X)_\bR$ and its induced operator norm,
\begin{equation}\label{eq:degree-norm-comparison}
\deg_i(f^n)
\asymp
\left\|(f^n)^*|_{\N^i(X)_\bR}\right\|
\asymp
\lambda_i(f)^n n^{\mu_i}
\qquad(n\to+\infty).
\end{equation}
Moreover, $\lambda_i(f)$ is itself an eigenvalue of $f^*|_{\N^i(X)_\bR}$, and it has a Jordan block of size $\mu_i+1$.
\end{proposition}

\begin{proof}
The cases $i=0,d$ are one-dimensional and immediate.
Fix $0<i<d$ and denote $M\coloneqq \dim_\bR \N^i(X)_\bR$.
We use the elementary fact that an interior point of a full-dimensional
cone is the sum of a basis contained in that cone.
By \cite[Lemma~2.12]{FL17a},
\[
\ell\coloneqq H^{d-i}\cap[X]
\in\operatorname{int}\overline{\Eff}_i(X),
\]
while $h\coloneqq H^i\in\operatorname{int}\PL^i(X)$ by
\cite[Lemma~3.14]{FL17a}.
Hence there are bases
$\{\alpha_j\}_{j=1}^M$ of $\N_i(X)_\bR$ and
$\{\beta_j\}_{j=1}^M$ of $\N^i(X)_\bR$ such that
\[
\alpha_j\in\overline{\Eff}_i(X),\qquad
\beta_j\in\PL^i(X),\qquad
\ell=\sum_j\alpha_j,\qquad h=\sum_j\beta_j.
\]
Define
\[
\|\beta\|_H\coloneqq\sum_{j=1}^M|\beta\cdot\alpha_j|,
\qquad \beta\in \N^i(X)_\bR.
\]
This is a norm on $\N^i(X)_\bR$, and for every nef dual class $\beta$,
\[
\|\beta\|_H=\beta\cdot \ell=\beta\cdot H^{d-i}.
\]

Fix an arbitrary $n\ge0$ and set $g=f^n$. Since
\[
g^*\PL^i(X)\subseteq\PL^i(X)\subseteq\Nef^i(X)
\]
by \cite[Remark~3.6 and Lemma~3.7]{FL17a}, both
$g^*\beta_j$ and $g^*(h-\beta_j)$ are nef. Therefore
\[
\|g^*\beta_j\|_H
=(g^*\beta_j)\cdot \ell
\le (g^*h)\cdot \ell
=\deg_i(g).
\]
For an arbitrary $\beta=\sum_j c_j\beta_j$, equivalence of the coefficient norm with
$\|\cdot\|_H$ gives
\[
\|g^*\beta\|_H
\le \deg_i(g)\sum_j|c_j|
\lesssim \deg_i(g)\|\beta\|_H.
\]
Thus $\bigl\|g^*|_{\N^i(X)_\bR}\bigr\|_H\lesssim \deg_i(g)$. Conversely,
\[
\deg_i(g)=\|g^*h\|_H
\le \bigl\|g^*|_{\N^i(X)_\bR}\bigr\|_H\|h\|_H
=H^d\,\bigl\|g^*|_{\N^i(X)_\bR}\bigr\|_H.
\]
Consequently,
\[
\deg_i(f^n)\asymp
\left\|(f^n)^*|_{\N^i(X)_\bR}\right\|,
\]
for any prescribed norm on $\N^i(X)_\bR$, with constants independent of $n$.

By functoriality, one has $(f^n)^*=(f^*)^n$ on $\N^i(X)_\bR$.
The preceding comparison and Gelfand's formula (see, e.g., \cite[Corollary~5.6.14]{HJ13}) give
\[
\lambda_i(f) \coloneqq \lim_{n\to+\infty}\deg_i(f^n)^{1/n}=\lim_{n\to+\infty}\left\|(f^*)^n|_{\N^i(X)_\bR}\right\|^{1/n}=\rho(f^*|_{\N^i(X)_\bR}),
\]
proving the existence and spectral interpretation of $\lambda_i(f)$.
Jordan normal form further yields
\[
\bigl\|(f^n)^*|_{\N^i(X)_\bR}\bigr\|\asymp \rho(f^*|_{\N^i(X)_\bR})^n n^{\mu_i}
=\lambda_i(f)^n n^{\mu_i}.
\]
Combining the last two comparisons proves \eqref{eq:degree-norm-comparison}.

Finally, by definition, $\PL^i(X)$ is a closed and convex cone in $\N^i(X)_\bR$; it is full-dimensional by \cite[Lemma~3.5]{FL17a}, salient by \cite[Lemma~3.7]{FL17a}, and $f^*$-invariant by \cite[Remark~3.6]{FL17a}.
Vandergraft's generalized Perron--Frobenius theorem \cite[Theorem~3.1]{Vandergraft68} shows that $\lambda_i(f)=\rho(f^*|_{\N^i(X)_\bR})$ is an eigenvalue of $f^*|_{\N^i(X)_\bR}$ and has a Jordan block of size $\mu_i+1$.
We thus complete the proof of \cref{prop:degree-norm-comparison}.
\end{proof}

\subsection{Proofs of the main results}

\begin{proof}[Proof of \cref{thm:endomorphism-peripheral-growth}]


As in \cref{ex:geometric-HL-models}(i), take
\[
V=\N^1(X)_\bR,
\qquad
\cC=\Amp(X),
\qquad
Q(D_1,\ldots,D_d)=D_1\cdots D_d,
\qquad
T=f^*.
\]
Then $(V, \cC, Q, T)$ is a dynamical Hodge--Lorentz package of degree $d\ge2$.
By \cref{prop:degree-norm-comparison},
\[
Q\bigl((T^nH)^{[i]},H^{[d-i]}\bigr)
=
\deg_i(f^n)
\asymp
\lambda_i(f)^n n^{\mu_i}.
\]
Thus the hypotheses of \cref{thm:abstract-turan-dichotomy} are
satisfied, and that theorem gives all the assertions.
\end{proof}


\begin{proof}[Proof of \cref{thm:zero-entropy-mixed}]
Take the same $(V,\cC,Q)$ in the proof of \cref{thm:endomorphism-peripheral-growth} as our Hodge--Lorentz package.
Since $\rho(f^*|_V)=1$ and $f^*$ acts invertibly on the integral N\'eron--Severi lattice modulo torsion, every eigenvalue of $f^*|_V$ is a root of unity by Kronecker's theorem.
Choose $m\ge1$ such that
\[
T\coloneqq (f^m)^*|_V=e^N
\]
is unipotent, where $N=\log T$ is nilpotent.
Then $(V, \cC, Q, T)$ is a unipotent dynamical Hodge--Lorentz package of degree $d\ge2$.

Fix a residue vector $\br=(r_1,\ldots,r_d)\in\{0,\ldots,m-1\}^d$ and set
\[
H_{j,r_j}\coloneqq(f^{r_j})^*H_j,
\qquad
\bH_{\br}\coloneqq(H_{1,r_1},\ldots,H_{d,r_d}).
\]
If $\bn=m\bq+\br$ with $\bq\in\bZ^d$, then
\[
\Phi_{f,\gamma,\bH}(\bn)
=
\Phi_{T,\gamma,\bH_{\br}}(\bq).
\]
Every $H_{j,r_j}$ is ample, so \cref{thm:unipotent-mixed-orbit-bound} shows that the right-hand side is a polynomial in $\bq$ of even total degree at most $b_\gamma$.
Consequently, on the residue class $\br$, the mixed orbit function agrees with the polynomial
\[
\bx
\longmapsto
\Phi_{T,\gamma,\bH_{\br}}
\left(\frac{\bx-\br}{m}\right).
\]
This polynomial has the same total degree.
This proves the quasipolynomial assertion for every $\bn\in\bZ^d$; the identities $e^{q_jN}=T^{q_j}$ for $q_j\in\bZ$ also cover negative coordinates.
The growth estimate follows from the finitely many residue classes, and restriction to $\bn=n\bl$ gives \eqref{eq:intro-zero-entropy-ray}.
\end{proof}


\begin{proof}[Proof of \cref{cor:zero-entropy-degree-growth}]
At zero entropy, $\lambda_i(f)=1$ for every $0\le i\le d$.
Up to replacing $f$ by a positive iterate, which does not change the largest size of Jordan blocks, we assume that $f^*|_{\N^1(X)_\bR}$ is unipotent.
Let $(V, \cC, Q, T)$ be the unipotent dynamical Hodge--Lorentz package of degree $d\ge2$; see \cref{ex:geometric-HL-models}(i).
Then as in the proof of \cref{thm:endomorphism-peripheral-growth}, by \cref{prop:degree-norm-comparison}, \cref{cor:zero-entropy-degree-growth} follows readily from \cref{thm:two-block-model} applied to the above $(V, \cC, Q, T)$.
\end{proof}


\begin{proof}[Proof of \cref{cor:kahler-growth}]
Since the automorphism $g$ has zero entropy, it is well known that all eigenvalues of $g^*$ on $H^\bullet(Y,\bC)$ are roots of unity by the Gromov--Yomdin theorem \cite{Gromov03,Yomdin87} and Dinh's estimate \cite{Dinh05}.
Up to replacing $g$ by a positive iterate, which does not change the largest size of Jordan blocks, we assume that
\[
T\coloneqq g^*|_{H^{1,1}(Y,\bR)}=e^N
\]
is unipotent.
Hence 
\[
\bigl(H^{1,1}(Y,\bR),\cK_Y,Q,T\bigr),
\qquad
Q(\alpha_1,\ldots,\alpha_d)
\coloneqq
\int_Y\alpha_1\wedge\cdots\wedge\alpha_d,
\]
is a unipotent dynamical Hodge--Lorentz package of degree $d\ge2$; see \cref{ex:geometric-HL-models}(ii).

Fix a K\"ahler class $\omega$.
For $0\le p\le d$, set
\[
\deg_p(g^n)
\coloneqq
\int_Y\bigl((g^n)^*\omega\bigr)^p\wedge\omega^{d-p}.
\]
Applying \cref{thm:two-block-model} with $h_1=h_2=\omega$ gives
\[
\deg_p(g^n) \asymp n^{\mu_p}=O\bigl(n^{2p(d-p)}\bigr).
\]
For every $p$, the standard mass--norm comparison for positive cohomology classes gives
\[
\left\|(g^n)^*|_{H^{p,p}(Y,\bC)}\right\|
\lesssim
\deg_p(g^n).
\]
Indeed, $\omega^p$ lies in the interior of the cone generated by classes of smooth closed strongly positive $(p,p)$-forms, pullback preserves this cone, and the same cone-basis argument as in \cref{prop:degree-norm-comparison} applies; see also \cite[Proof of Proposition~5.8]{Dinh05}.
The Hodge-theoretic Cauchy--Schwarz estimate used in the same proof then yields
\[
\left\|(g^n)^*|_{H^{p,q}(Y,\bC)}\right\|^2
\lesssim
\left\|(g^n)^*|_{H^{p,p}(Y,\bC)}\right\|
\left\|(g^n)^*|_{H^{q,q}(Y,\bC)}\right\|
\lesssim
\deg_p(g^n)\deg_q(g^n),
\]
where the implied constants are all independent of $n$;
see also \cite[Proof of Theorem~1.1]{DLOZ22}.
It follows that
\[
\left\|(g^n)^*|_{H^{p,q}(Y,\bC)}\right\|
=O\bigl(n^{p(d-p)+q(d-q)}\bigr).\qedhere
\]
\end{proof}

\appendix

\section{Sharpness on abelian varieties}\label{sec:abelian-sharpness}



This appendix constructs sharp examples on powers of elliptic curves: a single regular-unipotent automorphism simultaneously attains every zero-entropy mixed-orbit bound $b_\gamma$ and every codimension degree-growth bound.
Suitable isogenies realize the affine-interval bounds and the local defect $4$.
Over $\bC$, the Hodge-bidegree bounds are also sharp \cite[Remark~4.1]{DLOZ22}.
By Riemann--Roch, the proof reduces to mixed determinant calculations using Cauchy--Binet \cite[\S0.8.7, p.~28]{HJ13} and Krattenthaler's confluent Vandermonde formula \cite[Theorem~20]{Krattenthaler99}.
See also the characteristic-free polynomial-volume calculation in \cite[Theorem~1.1 and Appendix~A]{Hu-GK-AV}.

\begin{theorem}
\label{thm:sharpness-main}
Let $E$ be an elliptic curve over an algebraically closed field $\bk$, let
$A=E^d$, and let $u\colon A\to A$ be induced by
\[
U=I_d+J_d\in\GL_d(\bZ),
\]
where $J_d$ is the standard (upper-triangular) nilpotent Jordan block.
Then for every tuple $\bH=(H_1,\ldots,H_d)$ of ample divisors on $A$, the following hold.

\begin{enumerate}[label=\textnormal{(\arabic*)},ref=(\arabic*)]
\item \label{thm:abelian-mixed-sharpness} For every $\gamma\in\Comp(d)$,
$\Phi_{u,\gamma,\bH}$ extends to a polynomial satisfying
\begin{equation}\label{eq:theta-mixed-sharp}
\deg_{\tot}\Phi_{u,\gamma,\bH}=b_\gamma\coloneqq d^2-\|\gamma\|_2^2 = 2\sum_{1\le j<k\le d} \gamma_j\gamma_k.
\end{equation}
Moreover, if the entries of
$\bl\in\bZ^d$ are pairwise distinct on $\supp(\gamma)$, then
\begin{equation}\label{eq:theta-sharp-ray}
\Phi_{u,\gamma,\bH}(n\bl)\asymp n^{b_\gamma}.
\end{equation}
In particular, for every $0\le i\le d$ and ample divisor $H$ on $A$,
\begin{equation}\label{eq:degree-growth-sharp}
\deg_{i}(u^n)\coloneqq (u^n)^*H^i\cdot H^{d-i} \asymp n^{2i(d-i)}.
\end{equation}

\item \label{thm:abelian-affine-interval-sharpness} 
For every $0\le r<s\le d$, there is an isogeny $f_{r,s}$ of $E^d$ such that $[r,s]$ is a maximal affine interval of $( \log\lambda_i(f_{r,s}))_i$ and
\begin{equation}\label{eq:positive-affine-interval}
\mu_i(f_{r,s})=2(i-r)(s-i)
\qquad(r\le i\le s).
\end{equation}

\item \label{thm:abelian-hodge-sharpness} Assume that $\bk=\bC$.
For every $0\le p,q\le d$, one has
\begin{equation}\label{eq:hodge-norm-growth-sharp}
\left\|(u^n)^*|_{H^{p,q}(A,\bC)}\right\|
\asymp
n^{p(d-p)+q(d-q)}.
\end{equation}
Equivalently, the largest Jordan block of $u^*|_{H^{p,q}(A,\bC)}$ has size
\[
p(d-p)+q(d-q)+1.
\]
\end{enumerate}

Thus the mixed, degree, affine-interval, and Hodge-bidegree bounds, as well as the discrete Laplacian upper bound $4$, in \cref{thm:endomorphism-peripheral-growth,thm:zero-entropy-mixed,cor:zero-entropy-degree-growth,cor:kahler-growth}, are sharp.
\end{theorem}


Keep the notation of \cref{thm:sharpness-main}.
Let
\[
\Theta\coloneqq \sum_{i=1}^d \pr_i^*[0].
\]
which gives a principal polarization $\phi_{\Theta}\colon A\xrightarrow{\isom} \widehat A$.
The pullback $u^*|_{\N^1(A)_\bR}$ is also unipotent (see \cite[Proof of Theorem~1.8]{Hu-GK-AV} and \cite[Theorem~1.9]{Hu24}).
Denote
\[
L\coloneqq \log U \qquad\text{and}\qquad N\coloneqq \log(u^*|_{\N^1(A)_\bR}).
\]

For real $t$, define
\[
U^t\coloneqq e^{tL}\in \GL_d(\bR) \qquad\text{and}\qquad \Theta_t\coloneqq e^{tN}\Theta \in \N^1(A)_\bR.
\]

For every $\gamma\in\Comp(d)$, the canonical polynomial extension from \cref{def:mixed-orbit-functions} is
\[
\Phi_{u,\gamma,\boldsymbol{\Theta}}(\bt) \coloneqq (\Theta_{t_1})^{\gamma_1}\cdots (\Theta_{t_d})^{\gamma_d} \in \bR[\bt].
\]
The lemma below expresses $\Phi_{u,\gamma,\boldsymbol{\Theta}}$ as the sum of squares of certain determinants via the Riemann--Roch theorem and the Cauchy--Binet formula.

\begin{lemma}[Sum-of-squares formula]\label{lem:sum-square}
For every $\gamma\in\Comp(d)$ and $\bt\in\bR^d$,
\begin{equation}\label{eq:theta-sum-of-squares}
\frac{\Phi_{u,\gamma,\boldsymbol{\Theta}}(\bt)}{\gamma!}=
\sum_{\substack{\Gamma_j\subseteq\{1,\ldots,d\}\\[2pt] |\Gamma_j|=\gamma_j,\, 1\le j\le d}}\det
\begin{pmatrix}
(U^{t_1})_{\Gamma_1}\\
\vdots\\
(U^{t_d})_{\Gamma_d}
\end{pmatrix}^{\!2},
\end{equation}
where $(U^{t_j})_{\Gamma_j}$ denotes the submatrix formed by the rows indexed by $\Gamma_j$, and the corresponding block of rows is empty when $\gamma_j=0$.
\end{lemma}

\begin{proof}
Fix $\gamma\in \Comp(d)$.
Both sides of \eqref{eq:theta-sum-of-squares} are polynomials in $\bt$.
Indeed, $N$ and $L=\log U$ are nilpotent, so that $e^{tN}\Theta$ and
$U^t=e^{tL}$ depend polynomially on $t$.
It therefore suffices to prove the identity for $\bt=\bn\in\bZ_{\ge0}^d$.

Recall that $\Theta_n=(u^n)^*\Theta$.
For $\bx\in\bZ_{>0}^d$, set
\[
\Theta(\bx)\coloneqq \sum_{j=1}^d x_j\Theta_{n_j}.
\]
Under the product principal polarization $\Theta$, the Rosati involution is transpose on integral matrices.
Hence the endomorphism
\[
\phi_\Theta^{-1}\phi_{\Theta(\bx)}
=
\sum_{j=1}^d x_j(u^{n_j})^\dagger u^{n_j}
\]
is represented by the matrix
\[
S(\bx)\coloneqq \sum_{j=1}^d x_j(U^{n_j})^{\T}U^{n_j}\in \Mat_d(\bZ)\cap\GL_d(\bQ).
\]
The Riemann--Roch theorem \cite[\S16]{Mumford} and $\deg(M\colon E^d\to E^d)=(\det M)^2$ then yield
\[
\left(\frac{\Theta(\bx)^d}{d!}\right)^2 = \deg(\phi_{\Theta(\bx)})
= \deg(\phi_\Theta^{-1}\phi_{\Theta(\bx)}) = \det S(\bx)^2.
\]
Both quantities before squaring are positive, so
\begin{equation}\label{eq:intersection-determinant}
\frac{1}{d!}
\left(\sum_{j=1}^d x_j\Theta_{n_j}\right)^d
=
\det\left(\sum_{j=1}^d
x_j(U^{n_j})^{\T}U^{n_j}\right).
\end{equation}
Since both sides are polynomials in $\bx$, this identity holds identically.

Set
\[
\cU=
\begin{pmatrix}
\sqrt{x_1}U^{n_1}\\
\vdots\\
\sqrt{x_d}U^{n_d}
\end{pmatrix}
\in \Mat_{d^2\times d}(\bC).
\]
Since
\[
\overline{\cU}^{\T}\cU
=
\sum_{j=1}^d x_j(U^{n_j})^{\T}U^{n_j},
\]
the Cauchy--Binet formula \cite[\S0.8.7, p.~28]{HJ13} applied to $\overline{\cU}^{\T}\cU$ gives
\[
\det\Biggl(\sum_{j=1}^d
x_j(U^{n_j})^{\T}U^{n_j}\Biggr)
=
\sum_{\substack{\Gamma_j\subseteq\{1,\ldots,d\}\\[2pt]
                 \sum_j|\Gamma_j|=d}}
\left(\prod_{j=1}^d x_j^{|\Gamma_j|}\right)
\det
\begin{pmatrix}
(U^{n_1})_{\Gamma_1}\\
\vdots\\
(U^{n_d})_{\Gamma_d}
\end{pmatrix}^{\!2}.
\]
The coefficient of $\bx^\gamma$ on the left-hand side of \eqref{eq:intersection-determinant} is $\Phi_{u,\gamma,\boldsymbol\Theta}(\bn)/\gamma!$.
Comparing coefficients proves \eqref{eq:theta-sum-of-squares} for $\bn\in\bZ_{\ge0}^d$.
Polynomiality together with the Zariski density of $\bZ_{\ge0}^d$ completes the proof.
\end{proof}

\begin{proof}[Proof of \cref{thm:sharpness-main}\ref{thm:abelian-mixed-sharpness}]
Fix $\gamma\in\Comp(d)$ and set
\[
m_\gamma
\coloneqq
\sum_{1\le j<k\le d}\gamma_j\gamma_k
=
\frac12\bigl(d^2-\|\gamma\|_2^2\bigr)
=
\frac{b_\gamma}{2}.
\]

We first consider $\bH=\boldsymbol\Theta$.
For every $1\le j\le d$, take $\Gamma_j^\circ\coloneqq \{1,\ldots,\gamma_j\}$ and denote by $\Delta^\circ(\bt)$ the corresponding determinant in the sum-of-squares formula \eqref{eq:theta-sum-of-squares}.
We index the row and column of $(U^{t_j})_{\Gamma_j^\circ}$ by $1\le a_j\le \gamma_j$ and $1\le c\le d$, respectively.
Since for any $t_j\in \bR$,
\[
U^{t_j}=(I_d+J_d)^{t_j}=\sum_{r=0}^{d-1}\binom{t_j}{r}J_d^r,
\]
the $(a_j,c)$-entry of $(U^{t_j})_{\Gamma_j^\circ}$ is $\binom{t_j}{c-a_j}$ if $a_j\le c$; zero, otherwise.
Then in every nonzero term of the Leibniz expansion of $\Delta^\circ$, the maximal total degree is
\[
\sum_{c=1}^d c
-\sum_{j=1}^d\sum_{a_j=1}^{\gamma_j}a_j
=
\frac12\left(d^2-\sum_{j=1}^d\gamma_j^2\right)
=
m_\gamma.
\]
Consequently, the degree-$m_\gamma$ homogeneous part of $\Delta^\circ$ is obtained by replacing each binomial coefficient $\binom{t_j}{c-a_j}$ by its leading term:
\[
[\Delta^\circ]_{m_\gamma}
=
\det\left(
\frac{t_j^{\,c-a_j}}{(c-a_j)!}
\right)_{\substack{1\le j\le d,\ 1\le a_j\le\gamma_j\\[2pt]
1\le c\le d}},
\]
where the block structure is inherited from $\Delta^\circ$ and for every $1\le j\le d$, the $(a_j,c)$-entry is zero if $a_j>c$.
The confluent Vandermonde formula \cite[Theorem~20]{Krattenthaler99} gives
\begin{equation}\label{eq:sharp-confluent-vandermonde}
[\Delta^\circ]_{m_\gamma}
=
C_\gamma \prod_{1\le j<k\le d}(t_k-t_j)^{\gamma_j\gamma_k},
\qquad C_\gamma\neq0.
\end{equation}
Thus, if the entries of $\bl$ are pairwise distinct on
$\supp(\gamma)$,
\[
\Delta^\circ(n\bl)
\sim
C_\gamma
\prod_{1\le j<k\le d}(\ell_k-\ell_j)^{\gamma_j\gamma_k}
n^{m_\gamma},
\]
whose leading coefficient is nonzero. The sum-of-squares formula \eqref{eq:theta-sum-of-squares} therefore yields
\[
\Phi_{u,\gamma,\boldsymbol\Theta}(n\bl)
\ge
\gamma!\,\Delta^\circ(n\bl)^2
\gtrsim n^{b_\gamma}.
\]
Combining with \eqref{eq:intro-zero-entropy-ray} in \cref{thm:zero-entropy-mixed}, we get
\begin{equation}\label{eq:theta-ray-product}
\Phi_{u,\gamma,\boldsymbol\Theta}(n\bl)
\asymp n^{b_\gamma}.
\end{equation}

Now let $\bH=(H_1,\ldots,H_d)$ be arbitrary.
Choose $0<\epsilon<1$ sufficiently small such that $H_j-\epsilon\Theta$ and $\Theta-\epsilon H_j$ are ample for every $j$.
Since $u$ is an automorphism, pullback preserves the ample cone, and monotonicity of mixed intersections gives
\[
\epsilon^d\Phi_{u,\gamma,\boldsymbol\Theta}(\bn)
\le
\Phi_{u,\gamma,\bH}(\bn)
\le
\epsilon^{-d}\Phi_{u,\gamma,\boldsymbol\Theta}(\bn)
\qquad (\bn\in\bZ^d).
\]
Together with \eqref{eq:theta-ray-product}, this proves \eqref{eq:theta-sharp-ray}.
Taking, for instance, $\bl=(1,2,\ldots,d)$ shows that $\deg_{\tot}\Phi_{u,\gamma,\bH}\ge b_\gamma$; combined with \eqref{eq:intro-zero-entropy-mixed-bound} in \cref{thm:zero-entropy-mixed}, we prove \eqref{eq:theta-mixed-sharp}.

Finally, for $0<i<d$, take
\[
\gamma=(i,d-i,0,\ldots,0),\qquad
\bH=(H,\ldots,H),\qquad
\bl=(1,0,\ldots,0).
\]
Then
\[
\Phi_{u,\gamma,\bH}(n\bl)=\deg_i(u^n),
\qquad
b_\gamma=d^2-i^2-(d-i)^2=2i(d-i),
\]
so \eqref{eq:degree-growth-sharp} follows from \eqref{eq:theta-sharp-ray}.
For $i=0,d$, both sides are constant as $u$ is an automorphism.
\end{proof}

\begin{proof}[Proof of \cref{thm:sharpness-main}\ref{thm:abelian-affine-interval-sharpness}]
Fix $0\le r<s\le d$.
Set $k=s-r$ and $\ell=d-s$.
After omitting zero-dimensional factors, define
\[
f_{r,s} \coloneqq [5]_{E^r}\times\bigl([3]\circ u_k\bigr)_{E^k} \times[2]_{E^{\ell}},
\]
where $u_k$ is induced by $I_k+J_k$, with $u_1=\id_E$.
By \cref{thm:sharpness-main}\ref{thm:abelian-mixed-sharpness}, we compute the ordinary $i$-th degrees of $f_{r,s}$ with respect to the ample divisor $\Theta$.
After a cumbersome and tedious calculation on the product expansion, we have
\begin{equation}\label{eq:degree-growth-frs}
\deg_i(f_{r,s}^n) = \bigl((f_{r,s}^n)^*\Theta\bigr)^i \cdot \Theta^{d-i} \asymp
\sum_{\substack{a+b+c=i\\[2pt]
0\le a\le r,\;0\le b\le k,\;0\le c\le \ell}}
5^{2an}3^{2bn}2^{2cn}n^{2b(k-b)}.
\end{equation}
Indeed, write $\Theta = \Theta|_{E^r} + \Theta|_{E^k} + \Theta|_{E^\ell}$.
By commutativity and the fact that $[m]^*|_{\N^1(B)_\bR}=m^2\id$ for every $m\in\bZ$ and every abelian variety $B$,
\[
(f_{r,s}^n)^*\Theta = 5^{2n}\Theta|_{E^r} + 3^{2n} (u_k^n)^*\Theta|_{E^k} + 2^{2n}\Theta|_{E^\ell}.
\]
Hence
\[
\bigl((f_{r,s}^n)^*\Theta\bigr)^i = \sum_{\substack{a,b,c\ge0\\[2pt]
a+b+c=i}}
\binom{i}{a,b,c} 5^{2an}3^{2bn}2^{2cn}\, \Theta^a|_{E^r} \cdot (u_k^n)^*\Theta^b|_{E^k} \cdot \Theta^c|_{E^\ell}.
\]
On the other hand, after expanding $\Theta^{d-i} = (\Theta|_{E^r} + \Theta|_{E^k} + \Theta|_{E^\ell})^{d-i}$, only terms with complementary exponents $r-a$, $k-b$, and $\ell-c$ survive in the expansion of $\deg_i(f_{r,s}^n)$.
Applying \cref{thm:sharpness-main}\ref{thm:abelian-mixed-sharpness} to the automorphism $u_k$ of $E^k$, one has
\[
\deg_b(u_k^n) = (u_k^n)^*\Theta^b|_{E^k}\cdot\Theta^{k-b}|_{E^k} \asymp n^{2b(k-b)}.
\]
The asymptotic formula \eqref{eq:degree-growth-frs} follows with those multinomial coefficients absorbed into $\asymp$, that depend only on $r,s,d,i$, but not on $n$.

Since $5>3>2$, the exponentially dominant triple is obtained uniquely
by filling first the $5$-block, then the $3$-block, and finally the
$2$-block. Hence
\[
\lambda_i(f_{r,s})
=
\begin{cases}
5^{2i},&0\le i\le r,\\
5^{2r}3^{2(i-r)},&r\le i\le s,\\
5^{2r}3^{2k}2^{2(i-s)},&s\le i\le d.
\end{cases}
\]
Thus the three successive affine pieces of $(\log\lambda_i(f_{r,s}))_i$ have slopes
\[
2\log5,\qquad 2\log3,\qquad 2\log2,
\]
with the irrelevant pieces omitted when $r=0$ or $s=d$.
It follows that $[r,s]$ is a maximal affine interval of $(\log\lambda_i(f_{r,s}))_i$.

For $r\le i\le s$, the unique dominant triple in \eqref{eq:degree-growth-frs} is
$(a,b,c)=(r,i-r,0)$, and hence
\[
\deg_i(f_{r,s}^n)
\asymp
\bigl(5^{2r}3^{2(i-r)}\bigr)^n
n^{2(i-r)(k-i+r)}
=
\lambda_i(f_{r,s})^n
n^{2(i-r)(s-i)}.
\]
The degree--norm comparison, \cref{prop:degree-norm-comparison}, now proves \eqref{eq:positive-affine-interval}.
\end{proof}

\begin{proof}[Proof of \cref{thm:sharpness-main}\ref{thm:abelian-hodge-sharpness}]
The actions of $u^*$ on $H^{1,0}(A,\bC)$ and $H^{0,1}(A,\bC)$ are both regular unipotent and represented by $U$.
It is well known that the largest Jordan block of the $p$-th compound matrix $C_p(U)$, which represents the induced action on $H^{p,0}(A,\bC)=\bigwedge^p H^{1,0}(A,\bC)$, has size $p(d-p)+1$ (see, e.g., \cite[Theorem~IV]{Littlewood36}).
On the other hand, since
\[
H^{p,q}(A,\bC) = \bigwedge\nolimits^pH^{1,0}(A,\bC) \otimes \bigwedge\nolimits^qH^{0,1}(A,\bC),
\]
the largest Jordan block of $u^*|_{H^{p,q}(A,\bC)}$ has size
\[
(p(d-p)+1)+(q(d-q)+1)-1 = p(d-p)+q(d-q)+1.
\]
This is also a well-known fact about the Kronecker product (see \cite[Theorem~I]{Littlewood36}).
\end{proof}



\begin{thebibliography}{ATVdB90}

\bibitem[ATVdB90]{ATVdB90}
M.~Artin, J.~Tate, and M.~Van~den Bergh, \emph{Some algebras associated to automorphisms of elliptic curves}, The {G}rothendieck {F}estschrift, {V}ol. {I}, Progr. Math., vol.~86, Birkh\"{a}user Boston, Boston, MA, 1990, pp.~33--85. \MR{1086882}

\bibitem[AVdB90]{AVdB90}
M.~Artin and M.~Van~den Bergh, \emph{Twisted homogeneous coordinate rings}, J. Algebra \textbf{133} (1990), no.~2, 249--271. \MR{1067406}

\bibitem[BH20]{BH20}
Petter Br\"and\'en and June Huh, \emph{Lorentzian polynomials}, Ann. of Math. (2) \textbf{192} (2020), no.~3, 821--891. \MR{4172622}

\bibitem[BL26]{BL26}
Petter Br\"and\'en and Jonathan Leake, \emph{Lorentzian polynomials on cones}, Forum Math. Sigma \textbf{14} (2026), Paper No. e16, 34. \MR{5021615}

\bibitem[Bor91]{Borel91}
Armand Borel, \emph{Linear algebraic groups}, second ed., Graduate Texts in Mathematics, vol. 126, Springer-Verlag, New York, 1991. \MR{1102012}

\bibitem[CPR21]{CPR21}
Serge Cantat and Olga Paris-Romaskevich, \emph{Automorphisms of compact {K}\"{a}hler manifolds with slow dynamics}, Trans. Amer. Math. Soc. \textbf{374} (2021), no.~2, 1351--1389. \MR{4196396}

\bibitem[Dan20]{Dang20}
Nguyen-Bac Dang, \emph{Degrees of iterates of rational maps on normal projective varieties}, Proc. Lond. Math. Soc. (3) \textbf{121} (2020), no.~5, 1268--1310. \MR{4133708}

\bibitem[DF01]{DF01}
Jeffrey Diller and Charles Favre, \emph{Dynamics of bimeromorphic maps of surfaces}, Amer. J. Math. \textbf{123} (2001), no.~6, 1135--1169. \MR{1867314}

\bibitem[DF21]{DF21}
Nguyen-Bac Dang and Charles Favre, \emph{Spectral interpretations of dynamical degrees and applications}, Ann. of Math. (2) \textbf{194} (2021), no.~1, 299--359. \MR{4276288}

\bibitem[Din05]{Dinh05}
Tien-Cuong Dinh, \emph{Suites d'applications m\'{e}romorphes multivalu\'{e}es et courants laminaires}, J. Geom. Anal. \textbf{15} (2005), no.~2, 207--227. \MR{2152480}

\bibitem[DLOZ22]{DLOZ22}
Tien-Cuong Dinh, Hsueh-Yung Lin, Keiji Oguiso, and De-Qi Zhang, \emph{Zero entropy automorphisms of compact {K}\"{a}hler manifolds and dynamical filtrations}, Geom. Funct. Anal. \textbf{32} (2022), no.~3, 568--594. \MR{4431123}

\bibitem[DN06]{DN06}
Tien-Cuong Dinh and Vi\^et-Anh Nguy\^en, \emph{The mixed {H}odge-{R}iemann bilinear relations for compact {K}\"ahler manifolds}, Geom. Funct. Anal. \textbf{16} (2006), no.~4, 838--849. \MR{2255382}

\bibitem[DP04]{DP04}
Jean-Pierre Demailly and Mihai Paun, \emph{Numerical characterization of the {K}\"ahler cone of a compact {K}\"ahler manifold}, Ann. of Math. (2) \textbf{159} (2004), no.~3, 1247--1274. \MR{2113021}

\bibitem[DS05]{DS05a}
Tien-Cuong Dinh and Nessim Sibony, \emph{Une borne sup\'erieure pour l'entropie topologique d'une application rationnelle}, Ann. of Math. (2) \textbf{161} (2005), no.~3, 1637--1644. \MR{2180409}

\bibitem[FL17]{FL17a}
Mihai Fulger and Brian Lehmann, \emph{Positive cones of dual cycle classes}, Algebr. Geom. \textbf{4} (2017), no.~1, 1--28. \MR{3592463}

\bibitem[FW12]{FW12}
Charles Favre and Elizabeth Wulcan, \emph{Degree growth of monomial maps and {M}c{M}ullen's polytope algebra}, Indiana Univ. Math. J. \textbf{61} (2012), no.~2, 493--524. \MR{3043585}

\bibitem[Gro90]{Gromov90}
Mikha{\"{\i}}l Gromov, \emph{Convex sets and {K}\"ahler manifolds}, Advances in differential geometry and topology, World Sci. Publ., Teaneck, NJ, 1990, pp.~1--38. \MR{1095529}

\bibitem[Gro03]{Gromov03}
\bysame, \emph{On the entropy of holomorphic maps}, Enseign. Math. (2) \textbf{49} (2003), no.~3-4, 217--235, preprint SUNY (1977). \MR{2026895}

\bibitem[HHM{\etalchar{+}}25]{HHMWW}
Daoji Huang, June Huh, Mateusz Micha{\l}ek, Botong Wang, and Shouda Wang, \emph{Realizations of homology classes and projection areas}, preprint (2025), 40 pp., \arxiv{2505.08881v2}.

\bibitem[HJ13]{HJ13}
Roger~A. Horn and Charles~R. Johnson, \emph{Matrix analysis}, second ed., Cambridge University Press, Cambridge, 2013. \MR{2978290}

\bibitem[HJ25]{HJ25}
Fei Hu and Chen Jiang, \emph{An upper bound for polynomial volume growth of automorphisms of zero entropy}, to appear in Peking Math. J. (2025), 27 pp., \arxiv{2408.15804v2}, \doi{10.1007/s42543-025-00106-1}.

\bibitem[HJ26]{HJ-plov-lower}
\bysame, \emph{A lower bound for polynomial volume growth of automorphisms of zero entropy}, preprint (2026), 30 pp., \arxiv{2604.21398v2}.

\bibitem[Hu20]{Hu20a}
Fei Hu, \emph{A theorem of {T}its type for automorphism groups of projective varieties in arbitrary characteristic}, Math. Ann. \textbf{377} (2020), no.~3-4, 1573--1602, With an appendix by Tomohide Terasoma. \MR{4126902}

\bibitem[Hu24a]{Hu24}
\bysame, \emph{Eigenvalues and dynamical degrees of self-maps on abelian varieties}, J. Algebraic Geom. \textbf{33} (2024), no.~2, 265--293. \MR{4705374}

\bibitem[Hu24b]{Hu-GK-AV}
\bysame, \emph{Polynomial volume growth of quasi-unipotent automorphisms of abelian varieties (with an appendix in collaboration with {C}hen {J}iang)}, Int. Math. Res. Not. IMRN (2024), no.~8, 6374--6399. \MR{4735629}

\bibitem[Huh26]{Huh-SRI}
June Huh, \emph{Volume polynomials}, Proceedings of Symposia in Pure Mathematics, Algebraic Geometry: Fort Collins 2025, to appear, 24 pp., \arxiv{2601.13249v4}.

\bibitem[Hum75]{GTM21}
James~E. Humphreys, \emph{Linear algebraic groups}, Springer-Verlag, New York-Heidelberg, 1975, Graduate Texts in Mathematics, No. 21. \MR{0396773}

\bibitem[JL23]{JL23}
Chen Jiang and Zhiyuan Li, \emph{Algebraic reverse {K}hovanskii--{T}eissier inequality via {O}kounkov bodies}, Math. Z. \textbf{305} (2023), no.~2, Paper No. 26, 14 pp. \MR{4645768}

\bibitem[Kee00]{Keeler00}
Dennis~S. Keeler, \emph{Criteria for {$\sigma$}-ampleness}, J. Amer. Math. Soc. \textbf{13} (2000), no.~3, 517--532. \MR{1758752}

\bibitem[Kra99]{Krattenthaler99}
C.~Krattenthaler, \emph{Advanced determinant calculus}, S\'em. Lothar. Combin. \textbf{42} (1999), Art. B42q, 67, The Andrews Festschrift (Maratea, 1998). \MR{1701596}

\bibitem[LB19]{LB19}
Federico Lo~Bianco, \emph{On the cohomological action of automorphisms of compact {K}\"{a}hler threefolds}, Bull. Soc. Math. France \textbf{147} (2019), no.~3, 469--514. \MR{4030548}

\bibitem[Lit36]{Littlewood36}
D.~E. Littlewood, \emph{On induced and compound matrices}, Proc. London Math. Soc. (2) \textbf{40} (1936), no.~1, 370--381. \MR{1575831}

\bibitem[LOZ25]{LOZ25}
Hsueh-Yung Lin, Keiji Oguiso, and De-Qi Zhang, \emph{Polynomial log-volume growth and the {GK}-dimensions of twisted homogeneous coordinate rings}, J. Noncommut. Geom. \textbf{19} (2025), no.~2, 451--493. \MR{4886492}

\bibitem[LX16]{LX16}
Brian Lehmann and Jian Xiao, \emph{Convexity and {Z}ariski decomposition structure}, Geom. Funct. Anal. \textbf{26} (2016), no.~4, 1135--1189. \MR{3558307}

\bibitem[LX17]{LX17}
\bysame, \emph{Correspondences between convex geometry and complex geometry}, \'{E}pijournal Geom. Alg\'{e}brique \textbf{1} (2017), Art. 6, 29. \MR{3743109}

\bibitem[Mum70]{Mumford}
David Mumford, \emph{Abelian varieties}, Tata Institute of Fundamental Research Studies in Mathematics, vol.~5, Published for the Tata Institute of Fundamental Research, Bombay; Oxford University Press, London, 1970. \MR{0282985}

\bibitem[Van68]{Vandergraft68}
James~S. Vandergraft, \emph{Spectral properties of matrices which have invariant cones}, SIAM J. Appl. Math. \textbf{16} (1968), 1208--1222. \MR{244284}

\bibitem[Xia15]{Xiao15}
Jian Xiao, \emph{Weak transcendental holomorphic {M}orse inequalities on compact {K}\"{a}hler manifolds}, Ann. Inst. Fourier (Grenoble) \textbf{65} (2015), no.~3, 1367--1379. \MR{3449182}

\bibitem[Xie24]{Xie-SC}
Junyi Xie, \emph{Algebraic dynamics and recursive inequalities}, preprint (2024), 43 pp., \arxiv{2402.12678v2}.

\bibitem[Yom87]{Yomdin87}
Yosef Yomdin, \emph{Volume growth and entropy}, Israel J. Math. \textbf{57} (1987), no.~3, 285--300. \MR{889979}

\end{thebibliography}
\newcommand{\etalchar}[1]{$^{#1}$}
\providecommand{\bysame}{\leavevmode\hbox to3em{\hrulefill}\thinspace}

\end{document}